\documentclass[12pt,a4paper]{amsart}

\usepackage{amsbsy}
\usepackage{amsmath}

\usepackage{MnSymbol}
\usepackage{amsthm}
\usepackage{amsxtra}
\usepackage{braket}
\usepackage{comment}
\usepackage{graphics}
\usepackage{mathtools}

\usepackage{verbatim}

\usepackage{tikz}
\usetikzlibrary{patterns}
\usetikzlibrary{decorations.pathreplacing, decorations.markings}
\tikzset{->-/.style={decoration={markings,mark=at position #1 with {\arrow{>}}},postaction={decorate}}}

\usetikzlibrary{arrows}
\definecolor{qqwuqq}{rgb}{0,0.39215686274509803,0}
\definecolor{qqqqff}{rgb}{0,0,1}
\definecolor{uuuuuu}{rgb}{0.26666666666666666,0.26666666666666666,0.26666666666666666}
\definecolor{xdxdff}{rgb}{0.49019607843137253,0.49019607843137253,1}
\definecolor{ududff}{rgb}{0.30196078431372547,0.30196078431372547,1}
\definecolor{zzttqq}{rgb}{0.6,0.2,0}
\definecolor{qqwuqq}{rgb}{0,0.39215686274509803,0}
\definecolor{qqqqff}{rgb}{0,0,1}
\definecolor{uuuuuu}{rgb}{0.26666666666666666,0.26666666666666666,0.26666666666666666}
\definecolor{xdxdff}{rgb}{0.49019607843137253,0.49019607843137253,1}
\definecolor{ududff}{rgb}{0.30196078431372547,0.30196078431372547,1}

\usepackage[pdfborder={0 0 0}]{hyperref} 
\hypersetup{breaklinks, raiselinks}

\usepackage{bm}
\usepackage{verbatim}
\allowdisplaybreaks[4]

\usepackage[initials,lite]{amsrefs} 
\AtBeginDocument{%
   \def\MR#1{}
}

\usepackage{cleveref}

\usepackage{a4wide}
\theoremstyle{plain}
\newtheorem{Theorem}{Theorem}[section]
\newtheorem{Lemma}[Theorem]{Lemma}
\newtheorem{Corollary}[Theorem]{Corollary}
\newtheorem{Proposition}[Theorem]{Proposition}

\newtheorem{Observation}[Theorem]{Observation}

\theoremstyle{definition}
\newtheorem{Discussion}[Theorem]{Discussion}
\newtheorem{Assumptions and Discussion}[Theorem]{Assumptions and Discussion}

\newtheorem{Example}[Theorem]{Example}
\newtheorem{Definition}[Theorem]{Definition}

\newtheorem{Question}[Theorem]{Question}

\newtheorem{Remark}[Theorem]{Remark}

\theoremstyle{remark}

\newtheorem*{acknowledgment*}{Acknowledgment}

\usepackage[shortlabels]{enumitem}
\SetEnumerateShortLabel{a}{\textup{(\alph*)}}
\SetEnumerateShortLabel{A}{\textup{(\Alph*)}}
\SetEnumerateShortLabel{1}{\textup{(\arabic*)}}
\SetEnumerateShortLabel{i}{\textup{(\roman*)}}
\SetEnumerateShortLabel{I}{\textup{(\Roman*)}}

\def\bar#1{\overline{#1}}

\def\codim{\operatorname{codim}}

\def\deg{\operatorname{deg}}

\def\dim{\operatorname{dim}}

\def\gens{\operatorname{gens}}

\def\ini{\operatorname{in}} 

\def\ker{\operatorname{ker}}
\def\KK{{\mathbb K}}

\def\link{\operatorname{link}}
\def\lk{\operatorname{link}} 

\def\onto{\twoheadrightarrow}

\def\Sym{\operatorname{Sym}}

\def\ZZ{{\mathbb Z}}

\newcommand\bda{{\bm a}}

\newcommand\bdb{{\bm b}}

\newcommand\bdu{\bm{u}}
\newcommand\bdV{\bm{V}}
\newcommand\bdv{\bm{v}}

\newcommand\bdw{{\bm w}}

\newcommand\bdX{{\bm X}}

\newcommand\bfI{\mathbf{I}}

\newcommand\bfT{\mathbf{T}}

\newcommand\calA{\mathcal{A}}
\newcommand\calB{\mathcal{B}}
\newcommand\calC{\mathcal{C}}
\newcommand\calD{\mathcal{D}}
\newcommand\calE{\mathcal{E}}
\newcommand\calF{\mathcal{F}}
\newcommand\calG{\mathcal{G}}
\newcommand\calH{\mathcal{H}}
\newcommand\calI{\mathcal{I}}

\newcommand\calR{\mathcal{R}}
\newcommand\calS{\mathcal{S}}

\newcommand\calV{\mathcal{V}}
\newcommand\calZ{\mathcal{Z}}

\newcommand\frakb{\mathfrak{b}}

\newcommand\frakm{\mathfrak{m}}

\newcommand\frakp{\mathfrak{p}}

\newcommand\Spec{\operatorname{Spec}}

\newcommand{\Phan}{\operatorname{Phan}}

\usepackage{floatrow}
\usepackage{caption}
\DeclareCaptionSubType[alph]{figure}

\begin{document}
\baselineskip=16pt

\title[Koszul blowup algebras]{Koszul blowup algebras associated to three-dimensional Ferrers diagrams}

\author[Kuei-Nuan Lin, Yi-Huang Shen]{Kuei-Nuan Lin and Yi-Huang Shen}

\thanks{AMS 2010 {\em Mathematics Subject Classification}.
Primary 13F55, 13P10, 14M25, 16S37; Secondary 14M05, 14N10, 05E40, 05E45}


\address{Department of Mathematics, The Pennsylvania State University, Greater Allegheny, McKeesport, PA 15132}
\email{kul20@psu.edu}

\address{Key Laboratory of Wu Wen-Tsun Mathematics, Chinese Academy of Sciences, School of Mathematical Sciences, University of Science and Technology of China, Hefei, Anhui, 230026, P.R.~China}
\email{yhshen@ustc.edu.cn}



\begin{abstract}
    We investigate the Rees algebra and the toric ring of the squarefree monomial ideal associated to the three-dimensional Ferrers diagram. Under the projection property condition, we describe explicitly the presentation ideals of the Rees algebra and the toric ring. We show that the toric ring is a Koszul Cohen--Macaulay normal domain, while the Rees algebra is Koszul and the defining ideal is of fiber type. 
\end{abstract}

\maketitle
\section{Introduction}

Given a graded ideal $I$ in a standard graded ring $R$ over a field $\mathbb{K}$, one encounters the Rees algebra $\mathcal{R}(I)=R[It]$ of $I$, as well as the special fiber ring {$\mathcal{F}(I)=\mathcal{R}[I] \otimes_{R} \mathbb{K}$.}  These objects are important to commutative algebraists and geometers because the projective schemes of these rings define the blowup and the special fiber of the blowup of the scheme $\Spec(R)$ along $V(I)$ respectively.  The most challenging question of this topic is to describe those objects in term of generators and relationships, i.e., to find the presentation equations of these objects over some polynomial rings. When the ideal $I$ is generated by forms of the same degree, these rings describe the image and the graph of the rational map between the projective spaces.  The presentation equations of these algebras give implicit equations of the graph and of the variety parametrized by the map.  Finding those presentation equations is known as the implicitization problem \cite{MR2172855}.  When $I$ is a monomial ideal generated by the same degree in a polynomial over a field $\mathbb{K}$, the special fiber ring $\mathcal{F}(I)$ is the toric ring induced by $I$. The presentation ideal of $\mathcal{F}(I)$ is a prime binomial ideal, hence is a toric ideal by \cite[Proposition 1.1.11]{MR2810322}.  Toric ideals play an important role in polyhedral geometry, algebraic topology, algebraic geometry and statistics.  As pointed out in \cite{MR2039975}, even though it is known that the toric ideal is generated by binomials,  ``there are no simple formulas for a finite set of generators of a general toric ideal''.  And it is an active research area to understand and find the toric ideals; see for example, \cite{MR2433002}, \cite{MR2457403}, \cite{MR1705794}, \cite{MR3144398} and \cite{MR2492476}.

Finding the presentation ideal of $\mathcal{R}[I]$ when $I$ is a monomial ideal is another active field; see for example, \cite{arXiv:1205.3127} and \cite{MR1335312}.  Once we have the presentation ideal of $\mathcal{R}[I]$, we can obtain the presentation ideal of $\mathcal{F}(I)$ for free, simply because $\mathcal{R}[I]\otimes_{R}\mathbb{K=}\mathcal{F}(I)$. Of course, the reverse process is generally complicated.  

Ideals of \emph{fiber type} was introduced in \cite{MR2195995} for investigating Rees algebras. If an ideal $I$ is of fiber type, then the presentation ideal of $\calR[I]$ can be obtained from the combination of linear relations from the first syzygy of $I$, with the presentation ideal of $\calF(I)$; see also Definition \ref{fiber-type-def}.  There is no doubt that with respect to the Rees algebra $\calR(I)$, an ideal of fiber type provides the next best possibility if $I$ is not of linear type, i.e., $I$ is not defined by linear relations.  Since an explicit description of the presentation ideal of $\calR(I)$ is in general much involved and difficult, if $I$ is of fiber type, then the focus of investigation of $\calR(I)$ can be shifted to that of $\calF(I)$. This is also the strategy employed in this paper. Finding ideals of fiber type is another active field; see for example, \cite{MR2195995} and \cite{MR3697146}. 

Recall that if $I$ is a graded ideal in a polynomial ring $S$ over the field $\KK$, the quotient algebra $R=S/I$ is called \emph{Koszul} if the (in general infinite) minimal free resolution of $\KK$ over $R$ is linear.  For instance, Avramov and Eisenbud \cite{MR1195407} showed that if this $R$ is Koszul, then every graded finitely generated $R$-module has finite regularity over $R$.  Koszul property is probably the best possibility when one encounters infinite free resolutions. This even makes the Koszul property a basic notion of the representation theory in the non-commutative case.  As pointed out in \cite[Problem 34.6 and Question 74.1]{MR2560561}, what classes of toric rings are Koszul is an open and attractive question; see also \cite{MR3155951}. A related question is when projective toric varieties are defined by quadratics; see for example, \cite{MR2563140}, \cite{MR3003932} and \cite{MR3069289}. People are also interested in finding Cohen--Macaulay or normal toric rings.  See \cite{arXiv:1310.2496} for more on the Koszul algebras and \cite{MR1492542} for more on the toric variety and toric ideal. 

The main purpose of this work is to answer above open questions with respect to
the three-dimensional Ferrers diagram.  As pointed out in the work of Corso and
Nagel \cite{MR2457403}, ``Ferrers graph/tableaux have a prominent place in the
literature as they have been studied in relation to chromatic polynomials,
Schubert varieties, hypergeometric series, permutation statistics, quantum
mechanical operators, and inverse rook problems''; see \cite{MR2457403} for
detailed reference. It is known that the toric ring and the Rees
algebra associated to a (two-dimensional) Ferrers diagram are Cohen--Macaulay
normal domains; see \cite{MR1283294} or \cite{MR2457403}.  More
recently, the work of Corso, Nagel, Petrovi\'c, and Yuen \cite{CNPY} extends
the results to specialized Ferrers diagram and shows that the toric ring is a
Koszul Cohen-Macaulay normal domain. 

Interestingly, the special fiber rings of Ferrers diagrams can also be deemed
as the affine semigroup rings generated by some two-dimensional squarefree
monomials. In particular, they are isomorphic to the toric rings of incidence matrices of graphs. This kind of rings were well-studied by Hibi and Ohsugi. From this point of view, the special fiber rings of Ferrers diagrams are isomorphic to the toric rings of bipartite graph whose cycles of length $\ge 6$ has a chord. Consequently, the associated toric ideals admits a squarefree initial ideal by \cite[Theorem]{MR1657721}.

Since both papers \cite{MR2457403} and \cite{CNPY} involve monomial ideals
generated in degree two, it is natural to inquire about the degree three case.
As a result, we consider the three-dimensional Ferrers diagram and the monomial
ideal associated to it. Notice that the toric ideals of toric rings generated
by squarefree monomials of degree $\le 3$ are as complicated as any arbitrary toric ideals by \cite[Theorem 3.2]{arXiv:1711.04354}.

In some sense, the model we have here can be regarded as sub-configurations of
the 3-fold Segre product. From this point of view, a common strategy is a quest
for the existence of related algebra retracts. With that, properties like
normality of domains, regularity, complete intersection, Koszul,
Stanley-Reisner can descend along algebra retracts. However, this is not known
for properties like Cohen--Macaulay, Gorenstein in general.

Indeed, no such an algebra retract exists in general.  Unlike the
two-dimensional case, not every three-dimensional Ferrers diagram induces a
Koszul special fiber ring; see \Cref{MinExam}. This bad phenomenon happens due
to the recurrence of high-dimensional entanglements. Roughly speaking, in the
two-dimensional case, it is arguably easy to find extremal cells of the
diagram.  After deleting these extremal cells, one still gets a nice diagram of
similar configuration. On the other hand, in the three-dimensional case, unless
we have a cubic diagram, one can almost always expect an extremal cell of one
side being hampered by other sides of the diagram. The special cubic case is
essentially of matroidal type, hence has been investigated; see
\Cref{CubicExam}.

To circumvent the high-dimensional entanglements, we introduce the ``projection
property'' condition. Three-dimensional Ferrers diagrams with this condition
can be thought of as natural generalizations of two-dimensional Ferrers
diagrams.  See \Cref{why-shadow} for a heuristic explanation of this
introduction, as well as \Cref{why-shadow-grobner} for its usage.  Under the
projection property condition, we demonstrate that the toric ideal is generated
by quadratics.  Indeed, it has a quadratic Gr\"obner basis and the toric ring
is a Koszul Cohen--Macaulay normal domain.  We find out that the Ferrers ideal
satisfies the $\ell$-exchange property in the sense of \cite{MR2195995}. Hence,
the presentation ideal of the Rees algebra also have a quadratic Gr\"obner
basis and the Rees algebra is Koszul as well.  Moreover, the ideal is of fiber
type.

Here is the outline of this work. We start by setting the notations and definitions in Section 2. The main object is the generalized $2$-minors, $\bfI_{2}(\mathcal{D})$, that we propose as the generators of the toric ideal associated to the three-dimensional Ferrers diagram $\mathcal{D}$ (see \Cref{2-minors}). It is well-known that once the presentation ideal of an algebra has a quadratic Gr\"obner basis, the algebra is Koszul (see, for instance, \cite{MR2850142}). We show the set $\bfI_2(\calD)$ has a quadratic Gr\"obner basis if it comes from a Ferrers diagram satisfying the projection property (see \Cref{shadow} and \Cref{I2lq}). Not only that, we extend the quadratic Gr\"obner basis property to certain subdiagrams that we need for the later sections (see \Cref{I2-inheritance}).  Since the quadratic Gr\"obner basis of $\bfI_{2}(\mathcal{D})$ has a squarefree initial ideal, we can pass from the initial ideal to its Stanley-Reisner complex in Section 3. We demonstrate that the associated Stanley--Reisner complex is pure vertex-decomposable and hence is shellable. From this, we obtain the Cohen--Macaulayness of the ideal $\bfI_{2}(\mathcal{D})$ (see \Cref{VD-InductionOrder}). In Section 5, we show that the ideal $\bfI_{2}(\mathcal{D})$ is prime, by using the Cohen--Macaulay property of this ideal (see \Cref{I2prime}). Its proof is inspired by the work of Corso, Nagel, Petrovi\'c, and Yuen \cite{CNPY}. Finally, we put all pieces together to show that $\bfI_{2}(\mathcal{D})$ gives rise to the presentation ideal of the toric ideal, and the toric ring, namely the special fiber ring, is a Koszul Cohen--Macaulay normal domain (see \Cref{ToricIsKoszul}). With the $\ell$-exchange property of the ideal, we obtain the presentation ideals of the Rees algebra as well (see \Cref{L-exchange-def} and \Cref{ReesIdeal}). 

\section{Preliminaries}
In this section, we fix basic definitions and standard notations used
throughout the paper. More precisely, we define a set of binomials that sits
inside the presentation ideal of the special fiber ring that we are interested in.
Elementary properties of this set are provided. In particular, this set has a
quadratic Gr\"obner basis.

Throughout this paper, $\KK$ is a field of characteristic zero.  Let $\calD$ be
a diagram of finite lattice points in $\ZZ_+^3$.  Let 
\[
    m\coloneqq \max\Set{i:(i,j,k)\in\calD}, \quad
    n\coloneqq \max\Set{j:(i,j,k)\in\calD} \quad \text{and} \quad
    p\coloneqq \max\Set{k:(i,j,k)\in\calD}.
\] 
Associated to $\calD$ is the polynomial ring 
\[
    R=\KK[x_1,\dots,x_m,y_1,\dots,y_n,z_1,\dots,z_p]
\]
and the monomial ideal 
\[
    I_{\calD}\coloneqq (x_iy_jz_k:(i,j,k)\in \calD)\subset R.
\]
This ideal will be called the \emph{defining ideal} of $\calD$.

If we write $\frakm$ for the graded maximal ideal of $R$, the \emph{special
    fiber ring} of $I_{\calD}$ is
\[
    \calF(I_{\calD})\coloneqq \bigoplus_{l\ge 0} I_{\calD}^{l}/\frakm I_{\calD}^l \cong
    R[I_{\calD}t] \otimes_{R} R/\frakm.
\]
Sometimes, we also call it the \emph{toric ring} of $I_{\calD}$ and denote it
by $\KK[I_{\calD}]$. Let 
\[
    \KK[\bfT_{\calD}]\coloneqq \KK[T_{i,j,k}: (i,j,k)\in \calD]
\]
be the polynomial ring in the variables $T_{i,j,k}$ over the field $\KK$.  Consider the map $\varphi:\KK[\bfT_{\calD}]\to R$, given by $T_{i,j,k}\to x_iy_jz_k$, and extend algebraically. Then $\calF(I_{\calD})$ is canonically isomorphic to $\KK[\bfT_{\calD}]/\ker(\varphi)$. We will denote the kernel ideal $\ker(\varphi)$ by $J_{\calD}$ and call it the \emph{special fiber ideal} of $I_{\calD}$. Sometimes, we also call it the \emph{toric ideal} of $I_{\calD}$ and the \emph{presentation ideal} of $\calF(\calI_{\calD})$.  It is well-known that $J_{\calD}$ is a graded binomial ideal; see, for instance, \cite[Corollary 4.3]{MR1363949} or \cite{MR2611561}.  We also observe that $\calF(\calI_{\calD})$, being isomorphic to a subring of $R$, is a domain. Hence $J_{\calD}$ is a prime ideal.

Related and more complicated is the \emph{Rees algebra} of $I_{\calD}$:
\[
    \calR(I_{\calD})\coloneqq R[I_{\calD}t]=R\oplus I_{\calD}t\oplus
    I_{\calD}^2t^2\oplus \cdots \subset R[t].
\] 
Let 
\[
R[\bfT_{\calD}]\coloneqq R[T_{i,j,k}: (i,j,k)\in \calD]
\]
be the polynomial ring in the variables $T_{i,j,k}$ over $R$. Consider the map
$\psi:R[\bfT_{\calD}]\to R[t]$, given by $T_{i,j,k}\to x_iy_jz_kt$, and extend
algebraically. Then $\calR(I_{\calD})$ is canonically isomorphic to
$R[\bfT_{\calD}]/\ker(\psi)$. The kernel ideal $\ker(\psi)$ will be referred to
as the \emph{Rees ideal} of $I_{\calD}$, or the \emph{presentation ideal} of
$\calR(I_{\calD})$.

In fact, the epimorphism from $R[\bfT_{\calD}]$ to $\calR(I_\calD)$ factors through the symmetric algebra $\Sym(I_{\calD})$ as
\[
    R[\bfT_{\calD}] \stackrel{\alpha}{\onto} \Sym(I_{\calD})\stackrel{\beta}{\onto} \calR(I_\calD).
\]
When the epimorphism $\beta$ is indeed an isomorphism, the ideal $I_\calD$ will be called of \emph{linear type}. The next best possibility with respect to the presentation ideal of $\calR(I_{\calD})$ is when $I_{\calD}$ is of fiber type, a concept introduced by Herzog, Hibi and Vladoiu in \cite{MR2195995}. Using notations above, it can be formulated as follows.
\begin{Definition}
    \label{fiber-type-def}
    The defining ideal $I_\calD$ is called of \emph{fiber type} if 
    \[
        \ker (\psi)= \ker(\alpha)+\ker(\varphi)R[\bfT_{\calD}].
    \]
\end{Definition}

For the given diagram $\calD$, let 
\begin{align*}
    a_{\calD}\coloneqq \left|\Set{i:(i,j,k)\in\calD}\right|, \quad
    b_{\calD}\coloneqq \left|\Set{j:(i,j,k)\in\calD}\right| \text{ and }
    c_{\calD}\coloneqq \left|\Set{k:(i,j,k)\in\calD}\right|
\end{align*}
be the \emph{essential length, width and height} of $\calD$ respectively.  By
abuse of notation, for $\bdu=(i_0,j_0,k_0)\in \calD$, let
\[
a_{\calD}(\bdu)\coloneqq \max\Set{i: (i,j_0,k_0)\in \calD}.
\]
In a similar vein, we can define $b_{\calD}(\bdu)$ and $c_{\calD}(\bdu)$. 

\begin{Definition}
	\label{3F}
    Let $\calD$ be a diagram of finite lattice points in $\ZZ_+^3$. $\calD$ is
    called a \emph{three-dimensional Ferrers diagram} if for each
    $(i_0,j_0,k_0)\in \calD$, for every positive integers $i\le i_0$, $j\le
    j_{0}$ and $k\le k_0$, one has $(i,j,k)\in \calD$.
\end{Definition}

When the diagram $\calD$ is a Ferrers diagram, the defining ideal $I_{\calD}$ will also be called the \emph{Ferrers ideal} of $\calD$.

Obviously, if $\calD$ is a three-dimensional Ferrers diagram, then
$a_{\calD}=a_{\calD}(\bdu_0)$, $b_{\calD}=b_{\calD}(\bdu_0)$ and
$c_{\calD}=c_{\calD}(\bdu_0)$ for the point $\bdu_0=(1,1,1)$. In addition, if one of the three numbers is
one, then we get the classic two-dimensional Ferrers diagram.

Recall that a standard graded $\KK$-algebra $R$ is called \emph{Koszul} if the
residue class field $\KK=R/R_{+}$ has a linear $R$-resolution. If we can write
$R\cong S/J$ as the quotient of a polynomial ring $S$, to show the Koszulness
of $R$, it suffices to show that the homogeneous ideal $J$ has a quadratic
Gr\"obner basis with respect to some monomial order by \cite[Theorem
6.7]{MR2850142}. The aim of this paper is to find classes of three-dimensional
diagrams whose associated toric rings are Koszul. Therefore, we will focus on
those with quadratic Gr\"obner bases in some monomial order.

For a positive integer $n$, we denote by $[n]$ the finite set $\Set{1,2,\dots,n}$.

\begin{Example}
    \label{CubicExam}
    The easiest example is when $\calD$ is a full rectangular cylinder diagram,
    i.e., $\calD$ takes the form $[a_{\calD}]\times
    [b_{\calD}]\times[c_{\calD}]$. It follows from \cite[Theorem
    5.3(b)]{MR1957102} that the toric ring is Koszul in this situation.  
\end{Example}

Let $\bdV\colon \bdv_1,\dotsc,\bdv_m$ be a set of lattice points in
$\ZZ_+^3$.  The minimal three-dimensional Ferrers diagram containing $\bdV$
will be called \emph{the Ferrers diagram generated by $\bdV$}.

Unlike the two-dimensional case in \cite[Theorem 4.2]{CNPY}, not all toric
rings associated to three-dimensional Ferrers diagrams are Koszul.

\begin{Example}
    \label{MinExam}
    \Cref{mingenDegree3} provides a diagram $\calD$ generated by
    \[
        (1,2,3),(1,3,2),(2,1,3),(2,3,1),(3,1,2),(3,2,1).
    \]
    Therefore, it consists of the following lattice points 
    \begin{align*}
        (1,1,1),
        (1,1,2), 
        (1,1,3), 
        (1,2,1), 
        (1,2,2), 
        (1,2,3), 
        (1,3,1), 
        (1,3,2),\\ 
        (2,1,1), 
        (2,1,2), 
        (2,1,3), 
        (2,2,1), 
        (2,3,1), 
        (3,1,1), 
        (3,1,2), 
        (3,2,1). 
    \end{align*}
    This is a three-dimensional Ferrers diagram.  If we use the base ring 
    \begin{align*}
        S&=\KK[
        T_{1,1,1},
        T_{1,1,2}, 
        T_{1,1,3}, 
        T_{1,2,1}, 
        T_{1,2,2}, 
        T_{1,2,3}, 
        T_{1,3,1}, 
        T_{1,3,2}, \\
        & \quad\qquad  T_{2,1,1}, 
        T_{2,1,2}, 
        T_{2,1,3}, 
        T_{2,2,1}, 
        T_{2,3,1}, 
        T_{3,1,1}, 
        T_{3,1,2}, 
        T_{3,2,1}], 
    \end{align*}
    and apply the canonical epimorphism, the minimal generating set of the
    special fiber ideal $J_{\calD}$ contains the following degree three
    binomial
    \begin{align*}
       T_{1,2,3}T_{2,3,1}T_{3,1,2}-T_{1,3,2}T_{2,1,3}T_{3,2,1},
   \end{align*}
   as suggested by \texttt{Macaulay2} \cite{M2}. Thus, by \cite[Proposition
   6.3]{MR2850142}, the toric ring associated to $\calD$ cannot be Koszul. 

\end{Example}   

\begin{figure}[h]
    \begin{center}
        \begin{tikzpicture}[ every node/.style={scale=0.8}, line cap=round,line join=round,>=triangle 45,x=1cm,y=1cm]
            \clip(4.939366111847568,-6.174247535540005) rectangle (13.580633888152436,-0.14575246446002865);
            \fill[line width=2pt,color=qqqqff,fill=qqqqff,fill opacity=0.3] (8,-1) -- (9,-1.5) -- (8,-2) -- (7,-1.5) -- cycle;
            \fill[line width=2pt,color=qqqqff,fill=qqqqff,fill opacity=0.3] (7,-1.5) -- (7,-2.5) -- (8,-3) -- (8,-2) -- cycle;
            \fill[line width=2pt,color=qqqqff,fill=qqqqff,fill opacity=0.3] (8,-2) -- (9,-1.5) -- (9,-2.5) -- (8,-3) -- cycle;
            \fill[line width=2pt,color=qqqqff,fill=qqqqff,fill opacity=0.3] (11,-2.5) -- (12,-3) -- (11,-3.5) -- (10,-3) -- cycle;
            \fill[line width=2pt,color=qqqqff,fill=qqqqff,fill opacity=0.3] (12,-3) -- (12,-4) -- (11,-4.5) -- (11,-3.5) -- cycle;
            \fill[line width=2pt,color=qqqqff,fill=qqqqff,fill opacity=0.3] (10,-3) -- (11,-3.5) -- (11,-4.5) -- (10,-4) -- cycle;
            \fill[line width=2pt,color=qqqqff,fill=qqqqff,fill opacity=0.3] (8,-4) -- (9,-4.5) -- (8,-5) -- (7,-4.5) -- cycle;
            \fill[line width=2pt,color=qqqqff,fill=qqqqff,fill opacity=0.3] (9,-4.5) -- (9,-5.5) -- (8,-6) -- (8,-5) -- cycle;
            \fill[line width=2pt,color=qqqqff,fill=qqqqff,fill opacity=0.3] (7,-4.5) -- (8,-5) -- (8,-6) -- (7,-5.5) -- cycle;
            \fill[line width=2pt,color=qqwuqq,fill=qqwuqq,fill opacity=0.3] (10,-1) -- (11,-1.5) -- (10,-2) -- (9,-1.5) -- cycle;
            \fill[line width=2pt,color=qqwuqq,fill=qqwuqq,fill opacity=0.3] (9,-1.5) -- (9,-2.5) -- (10,-3) -- (10,-2) -- cycle;
            \fill[line width=2pt,color=qqwuqq,fill=qqwuqq,fill opacity=0.3] (10,-2) -- (11,-1.5) -- (11,-2.5) -- (10,-3) -- cycle;
            \fill[line width=2pt,color=qqwuqq,fill=qqwuqq,fill opacity=0.3] (7,-2.5) -- (8,-3) -- (7,-3.5) -- (6,-3) -- cycle;
            \fill[line width=2pt,color=qqwuqq,fill=qqwuqq,fill opacity=0.3] (8,-3) -- (8,-4) -- (7,-4.5) -- (7,-3.5) -- cycle;
            \fill[line width=2pt,color=qqwuqq,fill=qqwuqq,fill opacity=0.3] (6,-3) -- (7,-3.5) -- (7,-4.5) -- (6,-4) -- cycle;
            \fill[line width=2pt,color=qqwuqq,fill=qqwuqq,fill opacity=0.3] (10,-4) -- (11,-4.5) -- (10,-5) -- (9,-4.5) -- cycle;
            \fill[line width=2pt,color=qqwuqq,fill=qqwuqq,fill opacity=0.3] (9,-4.5) -- (9,-5.5) -- (10,-6) -- (10,-5) -- cycle;
            \fill[line width=2pt,color=qqwuqq,fill=qqwuqq,fill opacity=0.3] (10,-5) -- (11,-4.5) -- (11,-5.5) -- (10,-6) -- cycle;
            \draw [line width=2pt] (9,-0.5)-- (7,-1.5);
            \draw [line width=2pt] (7,-1.5)-- (7,-2.5);
            \draw [line width=2pt] (7,-2.5)-- (6,-3);
            \draw [line width=2pt] (6,-3)-- (6,-5);
            \draw [line width=2pt] (9,-0.5)-- (11,-1.5);
            \draw [line width=2pt] (11,-1.5)-- (11,-2.5);
            \draw [line width=2pt] (11,-2.5)-- (12,-3);
            \draw [line width=2pt] (12,-3)-- (12,-5);
            \draw [line width=2pt] (7,-1.5)-- (8,-2);
            \draw [line width=2pt] (8,-2)-- (10,-1);
            \draw [line width=2pt] (8,-1)-- (10,-2);
            \draw [line width=2pt] (10,-2)-- (11,-1.5);
            \draw [line width=2pt] (10,-2)-- (10,-4);
            \draw [line width=2pt] (8,-2)-- (8,-4);
            \draw [line width=2pt] (6,-3)-- (7,-3.5);
            \draw [line width=2pt] (7,-3.5)-- (9,-2.5);
            \draw [line width=2pt] (9,-2.5)-- (11,-3.5);
            \draw [line width=2pt] (9,-3.5)-- (7,-4.5);
            \draw [line width=2pt] (9,-3.5)-- (11,-4.5);
            \draw [line width=2pt] (6,-4)-- (8,-5);
            \draw [line width=2pt] (6,-5)-- (8,-6);
            \draw [line width=2pt] (11,-3.5)-- (11,-5.5);
            \draw [line width=2pt] (12,-5)-- (10,-6);
            \draw [line width=2pt] (10,-6)-- (10,-5);
            \draw [line width=2pt] (10,-5)-- (8,-4);
            \draw [line width=2pt] (8,-5)-- (10,-4);
            \draw [line width=2pt] (8,-5)-- (8,-6);
            \draw [line width=2pt] (8,-6)-- (9,-5.5);
            \draw [line width=2pt] (9,-5.5)-- (10,-6);
            \draw [line width=2pt] (9,-4.5)-- (9,-5.5);
            \draw [line width=2pt] (7,-3.5)-- (7,-5.5);
            \draw [line width=2pt] (7,-2.5)-- (8,-3);
            \draw [line width=2pt] (9,-1.5)-- (9,-2.5);
            \draw [line width=2pt] (10,-3)-- (11,-2.5);
            \draw [line width=2pt] (11,-3.5)-- (12,-3);
            \draw [line width=2pt] (10,-5)-- (12,-4);
            \draw [line width=2pt,color=qqqqff] (8,-1)-- (9,-1.5);
            \draw [line width=2pt,color=qqqqff] (9,-1.5)-- (8,-2);
            \draw [line width=2pt,color=qqqqff] (8,-2)-- (7,-1.5);
            \draw [line width=2pt,color=qqqqff] (7,-1.5)-- (8,-1);
            \draw [line width=2pt,color=qqqqff] (7,-1.5)-- (7,-2.5);
            \draw [line width=2pt,color=qqqqff] (7,-2.5)-- (8,-3);
            \draw [line width=2pt,color=qqqqff] (8,-3)-- (8,-2);
            \draw [line width=2pt,color=qqqqff] (8,-2)-- (7,-1.5);
            \draw [line width=2pt,color=qqqqff] (8,-2)-- (9,-1.5);
            \draw [line width=2pt,color=qqqqff] (9,-1.5)-- (9,-2.5);
            \draw [line width=2pt,color=qqqqff] (9,-2.5)-- (8,-3);
            \draw [line width=2pt,color=qqqqff] (8,-3)-- (8,-2);
            \draw [line width=2pt,color=qqqqff] (11,-2.5)-- (12,-3);
            \draw [line width=2pt,color=qqqqff] (12,-3)-- (11,-3.5);
            \draw [line width=2pt,color=qqqqff] (11,-3.5)-- (10,-3);
            \draw [line width=2pt,color=qqqqff] (10,-3)-- (11,-2.5);
            \draw [line width=2pt,color=qqqqff] (12,-3)-- (12,-4);
            \draw [line width=2pt,color=qqqqff] (12,-4)-- (11,-4.5);
            \draw [line width=2pt,color=qqqqff] (11,-4.5)-- (11,-3.5);
            \draw [line width=2pt,color=qqqqff] (11,-3.5)-- (12,-3);
            \draw [line width=2pt,color=qqqqff] (10,-3)-- (11,-3.5);
            \draw [line width=2pt,color=qqqqff] (11,-3.5)-- (11,-4.5);
            \draw [line width=2pt,color=qqqqff] (11,-4.5)-- (10,-4);
            \draw [line width=2pt,color=qqqqff] (10,-4)-- (10,-3);
            \draw [line width=2pt,color=qqqqff] (8,-4)-- (9,-4.5);
            \draw [line width=2pt,color=qqqqff] (9,-4.5)-- (8,-5);
            \draw [line width=2pt,color=qqqqff] (8,-5)-- (7,-4.5);
            \draw [line width=2pt,color=qqqqff] (7,-4.5)-- (8,-4);
            \draw [line width=2pt,color=qqqqff] (9,-4.5)-- (9,-5.5);
            \draw [line width=2pt,color=qqqqff] (9,-5.5)-- (8,-6);
            \draw [line width=2pt,color=qqqqff] (8,-6)-- (8,-5);
            \draw [line width=2pt,color=qqqqff] (8,-5)-- (9,-4.5);
            \draw [line width=2pt,color=qqqqff] (7,-4.5)-- (8,-5);
            \draw [line width=2pt,color=qqqqff] (8,-5)-- (8,-6);
            \draw [line width=2pt,color=qqqqff] (8,-6)-- (7,-5.5);
            \draw [line width=2pt,color=qqqqff] (7,-5.5)-- (7,-4.5);
            \draw [line width=2pt] (9,-2.5)-- (9,-3.5);
            \draw [line width=2pt,color=qqwuqq] (10,-1)-- (11,-1.5);
            \draw [line width=2pt,color=qqwuqq] (11,-1.5)-- (10,-2);
            \draw [line width=2pt,color=qqwuqq] (10,-2)-- (9,-1.5);
            \draw [line width=2pt,color=qqwuqq] (9,-1.5)-- (10,-1);
            \draw [line width=2pt,color=qqwuqq] (9,-1.5)-- (9,-2.5);
            \draw [line width=2pt,color=qqwuqq] (9,-2.5)-- (10,-3);
            \draw [line width=2pt,color=qqwuqq] (10,-3)-- (10,-2);
            \draw [line width=2pt,color=qqwuqq] (10,-2)-- (9,-1.5);
            \draw [line width=2pt,color=qqwuqq] (10,-2)-- (11,-1.5);
            \draw [line width=2pt,color=qqwuqq] (11,-1.5)-- (11,-2.5);
            \draw [line width=2pt,color=qqwuqq] (11,-2.5)-- (10,-3);
            \draw [line width=2pt,color=qqwuqq] (10,-3)-- (10,-2);
            \draw [line width=2pt,color=qqwuqq] (7,-2.5)-- (8,-3);
            \draw [line width=2pt,color=qqwuqq] (8,-3)-- (7,-3.5);
            \draw [line width=2pt,color=qqwuqq] (7,-3.5)-- (6,-3);
            \draw [line width=2pt,color=qqwuqq] (6,-3)-- (7,-2.5);
            \draw [line width=2pt,color=qqwuqq] (8,-3)-- (8,-4);
            \draw [line width=2pt,color=qqwuqq] (8,-4)-- (7,-4.5);
            \draw [line width=2pt,color=qqwuqq] (7,-4.5)-- (7,-3.5);
            \draw [line width=2pt,color=qqwuqq] (7,-3.5)-- (8,-3);
            \draw [line width=2pt,color=qqwuqq] (6,-3)-- (7,-3.5);
            \draw [line width=2pt,color=qqwuqq] (7,-3.5)-- (7,-4.5);
            \draw [line width=2pt,color=qqwuqq] (7,-4.5)-- (6,-4);
            \draw [line width=2pt,color=qqwuqq] (6,-4)-- (6,-3);
            \draw [line width=2pt,color=qqwuqq] (10,-4)-- (11,-4.5);
            \draw [line width=2pt,color=qqwuqq] (11,-4.5)-- (10,-5);
            \draw [line width=2pt,color=qqwuqq] (10,-5)-- (9,-4.5);
            \draw [line width=2pt,color=qqwuqq] (9,-4.5)-- (10,-4);
            \draw [line width=2pt,color=qqwuqq] (9,-4.5)-- (9,-5.5);
            \draw [line width=2pt,color=qqwuqq] (9,-5.5)-- (10,-6);
            \draw [line width=2pt,color=qqwuqq] (10,-6)-- (10,-5);
            \draw [line width=2pt,color=qqwuqq] (10,-5)-- (9,-4.5);
            \draw [line width=2pt,color=qqwuqq] (10,-5)-- (11,-4.5);
            \draw [line width=2pt,color=qqwuqq] (11,-4.5)-- (11,-5.5);
            \draw [line width=2pt,color=qqwuqq] (11,-5.5)-- (10,-6);
            \draw [line width=2pt,color=qqwuqq] (10,-6)-- (10,-5);
            \draw (7.3,-1.2) node[anchor=north west] {\textbf{(2,1,3)}};
            \draw (9.4,-1.2) node[anchor=north west] {\textbf{(1,2,3)}};
            \draw (5.8,-3.4) node[anchor=north west] {\textbf{(3,1,2)}};
            \draw (10.8,-3.47) node[anchor=north west] {\textbf{(1,3,2)}};
            \draw (6.8,-4.918842113761193) node[anchor=north west] {\textbf{(3,2,1)}};
            \draw (9.8,-5) node[anchor=north west] {\textbf{(2,3,1)}};
            \begin{scriptsize}
                \draw [fill=ududff] (9,-0.5) circle (2.5pt);
                \draw [fill=ududff] (7,-1.5) circle (2.5pt);
                \draw [fill=ududff] (7,-2.5) circle (2.5pt);
                \draw [fill=ududff] (6,-3) circle (2.5pt);
                \draw [fill=ududff] (6,-5) circle (2.5pt);
                \draw [fill=ududff] (11,-1.5) circle (2.5pt);
                \draw [fill=ududff] (11,-2.5) circle (2.5pt);
                \draw [fill=ududff] (12,-3) circle (2.5pt);
                \draw [fill=ududff] (12,-5) circle (2.5pt);
                \draw [fill=ududff] (8,-2) circle (2.5pt);
                \draw [fill=xdxdff] (10,-1) circle (2.5pt);
                \draw [fill=xdxdff] (8,-1) circle (2.5pt);
                \draw [fill=ududff] (10,-2) circle (2.5pt);
                \draw [fill=ududff] (10,-4) circle (2.5pt);
                \draw [fill=ududff] (8,-4) circle (2.5pt);
                \draw [fill=ududff] (7,-3.5) circle (2.5pt);
                \draw [fill=ududff] (9,-2.5) circle (2.5pt);
                \draw [fill=ududff] (11,-3.5) circle (2.5pt);
                \draw [fill=ududff] (9,-3.5) circle (2.5pt);
                \draw [fill=ududff] (7,-4.5) circle (2.5pt);
                \draw [fill=ududff] (11,-4.5) circle (2.5pt);
                \draw [fill=xdxdff] (6,-4) circle (2.5pt);
                \draw [fill=ududff] (8,-5) circle (2.5pt);
                \draw [fill=ududff] (8,-6) circle (2.5pt);
                \draw [fill=ududff] (11,-5.5) circle (2.5pt);
                \draw [fill=ududff] (10,-6) circle (2.5pt);
                \draw [fill=ududff] (10,-5) circle (2.5pt);
                \draw [fill=ududff] (9,-5.5) circle (2.5pt);
                \draw [fill=ududff] (9,-4.5) circle (2.5pt);
                \draw [fill=xdxdff] (7,-5.5) circle (2.5pt);
                \draw [fill=uuuuuu] (8,-3) circle (2pt);
                \draw [fill=xdxdff] (9,-1.5) circle (2.5pt);
                \draw [fill=xdxdff] (10,-3) circle (2.5pt);
                \draw [fill=xdxdff] (12,-4) circle (2.5pt);
            \end{scriptsize}
        \end{tikzpicture}
        \caption{The minimal diagram $\calD$ with a degree $3$ minimal generator}
        \label{mingenDegree3} 
    \end{center}
\end{figure}
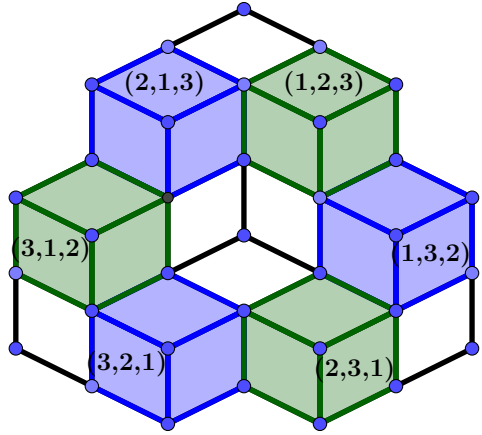

Meanwhile, one can easily find abundant three-dimensional Ferrers diagrams
whose associated special fiber ideals are quadratic, while their Gr\"obner
bases in the common monomial orders are not quadratic. To circumvent the
recurrence of high-dimensional entanglement when treating the Gr\"obner
basis, we introduce the following projection property for three-dimensional Ferrers
diagrams.  We will show that three-dimensional Ferrers diagram which satisfies
the projection property will have a Koszul associated toric ring in
\Cref{ToricIsKoszul}.

\begin{Definition} 
    \label{shadow}
    Let $\calD$ be a three-dimensional Ferrers diagram.  For each $2\le i \le a_{\calD}$, let $b_i=\max\Set{j: (i,j,1)\in \calD}$ and $c_i=\max\Set{k: (i,1,k)\in \calD}$. 
    Then the \emph{projection} of the $x=i$ layer is the set
    \[
        \Set{(i-1,j,k)\in \ZZ_+^3: j\le b_i \text{ and }k\le c_i}.
    \]
    And $\calD$ is said to satisfy the \emph{projection property} if the $x=i-1$ layer covers the projection of the $x=i$ layer for $2\le i \le a_{\calD}$, i.e., the following equivalent conditions hold:
    \begin{enumerate}[a]
        \item $(i-1,b_i,c_i)\in \calD$;
        \item if $(i,j_1,k_1)\in \calD$ and $(i,j_2,k_2)\in \calD$, then $(i-1,j_1,k_2)\in \calD$.
    \end{enumerate}
\end{Definition}

Trivially, the cubic diagram in \Cref{CubicExam} satisfies the projection
property.  On the other hand, for the diagram $\calD$ in \Cref{MinExam}, one
has $b_2=c_2=3$. Since $(1,3,3)\notin \calD$, the diagram $\calD$ does not
satisfy the projection property. One can attach $(2,2,2)$ and $(1,3,3)$ to
get the ``closure diagram'' $\overline{\calD}$ with respect to the projection
property, which is illustrated in \Cref{Shadow}.

\begin{figure}[h]
    \begin{center} 
        \begin{tikzpicture}[every node/.style={scale=0.8}, line cap=round,line join=round,>=triangle 45,x=1cm,y=1cm]
            \clip(4.939366111847569,-6.034049975747447) rectangle (13.580633888152443,-0.14575246446002865);
            \fill[line width=2pt,color=qqqqff,fill=qqqqff,fill opacity=0.3] (8,-1) -- (9,-1.5) -- (8,-2) -- (7,-1.5) -- cycle;
            \fill[line width=2pt,color=qqqqff,fill=qqqqff,fill opacity=0.3] (7,-1.5) -- (7,-2.5) -- (8,-3) -- (8,-2) -- cycle;
            \fill[line width=2pt,color=qqqqff,fill=qqqqff,fill opacity=0.3] (8,-2) -- (9,-1.5) -- (9,-2.5) -- (8,-3) -- cycle;
            \fill[line width=2pt,color=qqqqff,fill=qqqqff,fill opacity=0.3] (12,-3) -- (12,-4) -- (11,-4.5) -- (11,-3.5) -- cycle;
            \fill[line width=2pt,color=qqqqff,fill=qqqqff,fill opacity=0.3] (10,-3) -- (11,-3.5) -- (11,-4.5) -- (10,-4) -- cycle;
            \fill[line width=2pt,color=qqqqff,fill=qqqqff,fill opacity=0.3] (8,-4) -- (9,-4.5) -- (8,-5) -- (7,-4.5) -- cycle;
            \fill[line width=2pt,color=qqqqff,fill=qqqqff,fill opacity=0.3] (9,-4.5) -- (9,-5.5) -- (8,-6) -- (8,-5) -- cycle;
            \fill[line width=2pt,color=qqqqff,fill=qqqqff,fill opacity=0.3] (7,-4.5) -- (8,-5) -- (8,-6) -- (7,-5.5) -- cycle;
            \fill[line width=2pt,color=qqwuqq,fill=qqwuqq,fill opacity=0.3] (10,-1) -- (11,-1.5) -- (10,-2) -- (9,-1.5) -- cycle;
            \fill[line width=2pt,color=qqwuqq,fill=qqwuqq,fill opacity=0.3] (9,-1.5) -- (9,-2.5) -- (10,-3) -- (10,-2) -- cycle;
            \fill[line width=2pt,color=qqwuqq,fill=qqwuqq,fill opacity=0.3] (7,-2.5) -- (8,-3) -- (7,-3.5) -- (6,-3) -- cycle;
            \fill[line width=2pt,color=qqwuqq,fill=qqwuqq,fill opacity=0.3] (8,-3) -- (8,-4) -- (7,-4.5) -- (7,-3.5) -- cycle;
            \fill[line width=2pt,color=qqwuqq,fill=qqwuqq,fill opacity=0.3] (6,-3) -- (7,-3.5) -- (7,-4.5) -- (6,-4) -- cycle;
            \fill[line width=2pt,color=qqwuqq,fill=qqwuqq,fill opacity=0.3] (10,-4) -- (11,-4.5) -- (10,-5) -- (9,-4.5) -- cycle;
            \fill[line width=2pt,color=qqwuqq,fill=qqwuqq,fill opacity=0.3] (9,-4.5) -- (9,-5.5) -- (10,-6) -- (10,-5) -- cycle;
            \fill[line width=2pt,color=qqwuqq,fill=qqwuqq,fill opacity=0.3] (10,-5) -- (11,-4.5) -- (11,-5.5) -- (10,-6) -- cycle;
            \fill[line width=2pt,color=zzttqq,fill=zzttqq,fill opacity=0.8] (10,-2) -- (11,-1.5) -- (12,-2) -- (11,-2.5) -- cycle;
            \fill[line width=2pt,color=zzttqq,fill=zzttqq,fill opacity=0.75] (10,-2) -- (10,-3) -- (11,-3.5) -- (11,-2.5) -- cycle;
            \fill[line width=2pt,color=zzttqq,fill=zzttqq,fill opacity=0.75] (11,-2.5) -- (12,-2) -- (12,-3) -- (11,-3.5) -- cycle;
            \fill[line width=2pt,color=zzttqq,fill=zzttqq,fill opacity=0.75] (9,-2.5) -- (10,-3) -- (9,-3.5) -- (8,-3) -- cycle;
            \fill[line width=2pt,color=zzttqq,fill=zzttqq,fill opacity=0.75] (10,-3) -- (10,-4) -- (9,-4.5) -- (9,-3.5) -- cycle;
            \fill[line width=2pt,color=zzttqq,fill=zzttqq,fill opacity=0.75] (8,-3) -- (9,-3.5) -- (9,-4.5) -- (8,-4) -- cycle;
            \draw [line width=2pt] (9,-0.5)-- (7,-1.5);
            \draw [line width=2pt] (7,-1.5)-- (7,-2.5);
            \draw [line width=2pt] (7,-2.5)-- (6,-3);
            \draw [line width=2pt] (6,-3)-- (6,-5);
            \draw [line width=2pt] (9,-0.5)-- (11,-1.5);
            \draw [line width=2pt] (12,-3)-- (12,-5);
            \draw [line width=2pt] (7,-1.5)-- (8,-2);
            \draw [line width=2pt] (8,-2)-- (10,-1);
            \draw [line width=2pt] (8,-1)-- (10,-2);
            \draw [line width=2pt] (10,-2)-- (11,-1.5);
            \draw [line width=2pt] (10,-2)-- (10,-4);
            \draw [line width=2pt] (8,-2)-- (8,-4);
            \draw [line width=2pt] (6,-3)-- (7,-3.5);
            \draw [line width=2pt] (7,-3.5)-- (9,-2.5);
            \draw [line width=2pt] (9,-2.5)-- (11,-3.5);
            \draw [line width=2pt] (6,-4)-- (8,-5);
            \draw [line width=2pt] (6,-5)-- (8,-6);
            \draw [line width=2pt] (11,-3.5)-- (11,-5.5);
            \draw [line width=2pt] (12,-5)-- (10,-6);
            \draw [line width=2pt] (10,-6)-- (10,-5);
            \draw [line width=2pt] (10,-5)-- (8,-4);
            \draw [line width=2pt] (8,-5)-- (10,-4);
            \draw [line width=2pt] (8,-5)-- (8,-6);
            \draw [line width=2pt] (8,-6)-- (9,-5.5);
            \draw [line width=2pt] (9,-5.5)-- (10,-6);
            \draw [line width=2pt] (9,-4.5)-- (9,-5.5);
            \draw [line width=2pt] (7,-3.5)-- (7,-5.5);
            \draw [line width=2pt] (7,-2.5)-- (8,-3);
            \draw [line width=2pt] (9,-1.5)-- (9,-2.5);
            \draw [line width=2pt] (11,-3.5)-- (12,-3);
            \draw [line width=2pt] (10,-5)-- (12,-4);
            \draw [line width=2pt,color=qqqqff] (8,-1)-- (9,-1.5);
            \draw [line width=2pt,color=qqqqff] (9,-1.5)-- (8,-2);
            \draw [line width=2pt,color=qqqqff] (8,-2)-- (7,-1.5);
            \draw [line width=2pt,color=qqqqff] (7,-1.5)-- (8,-1);
            \draw [line width=2pt,color=qqqqff] (7,-1.5)-- (7,-2.5);
            \draw [line width=2pt,color=qqqqff] (7,-2.5)-- (8,-3);
            \draw [line width=2pt,color=qqqqff] (8,-3)-- (8,-2);
            \draw [line width=2pt,color=qqqqff] (8,-2)-- (7,-1.5);
            \draw [line width=2pt,color=qqqqff] (8,-2)-- (9,-1.5);
            \draw [line width=2pt,color=qqqqff] (9,-1.5)-- (9,-2.5);
            \draw [line width=2pt,color=qqqqff] (9,-2.5)-- (8,-3);
            \draw [line width=2pt,color=qqqqff] (8,-3)-- (8,-2);
            \draw [line width=2pt,color=qqqqff] (12,-3)-- (12,-4);
            \draw [line width=2pt,color=qqqqff] (12,-4)-- (11,-4.5);
            \draw [line width=2pt,color=qqqqff] (11,-4.5)-- (11,-3.5);
            \draw [line width=2pt,color=qqqqff] (11,-3.5)-- (12,-3);
            \draw [line width=2pt,color=qqqqff] (10,-3)-- (11,-3.5);
            \draw [line width=2pt,color=qqqqff] (11,-3.5)-- (11,-4.5);
            \draw [line width=2pt,color=qqqqff] (11,-4.5)-- (10,-4);
            \draw [line width=2pt,color=qqqqff] (10,-4)-- (10,-3);
            \draw [line width=2pt,color=qqqqff] (8,-4)-- (9,-4.5);
            \draw [line width=2pt,color=qqqqff] (9,-4.5)-- (8,-5);
            \draw [line width=2pt,color=qqqqff] (8,-5)-- (7,-4.5);
            \draw [line width=2pt,color=qqqqff] (7,-4.5)-- (8,-4);
            \draw [line width=2pt,color=qqqqff] (9,-4.5)-- (9,-5.5);
            \draw [line width=2pt,color=qqqqff] (9,-5.5)-- (8,-6);
            \draw [line width=2pt,color=qqqqff] (8,-6)-- (8,-5);
            \draw [line width=2pt,color=qqqqff] (8,-5)-- (9,-4.5);
            \draw [line width=2pt,color=qqqqff] (7,-4.5)-- (8,-5);
            \draw [line width=2pt,color=qqqqff] (8,-5)-- (8,-6);
            \draw [line width=2pt,color=qqqqff] (8,-6)-- (7,-5.5);
            \draw [line width=2pt,color=qqqqff] (7,-5.5)-- (7,-4.5);
            \draw [line width=2pt,color=qqwuqq] (10,-1)-- (11,-1.5);
            \draw [line width=2pt,color=qqwuqq] (11,-1.5)-- (10,-2);
            \draw [line width=2pt,color=qqwuqq] (10,-2)-- (9,-1.5);
            \draw [line width=2pt,color=qqwuqq] (9,-1.5)-- (10,-1);
            \draw [line width=2pt,color=qqwuqq] (9,-1.5)-- (9,-2.5);
            \draw [line width=2pt,color=qqwuqq] (9,-2.5)-- (10,-3);
            \draw [line width=2pt,color=qqwuqq] (10,-3)-- (10,-2);
            \draw [line width=2pt,color=qqwuqq] (10,-2)-- (9,-1.5);
            \draw [line width=2pt,color=qqwuqq] (7,-2.5)-- (8,-3);
            \draw [line width=2pt,color=qqwuqq] (8,-3)-- (7,-3.5);
            \draw [line width=2pt,color=qqwuqq] (7,-3.5)-- (6,-3);
            \draw [line width=2pt,color=qqwuqq] (6,-3)-- (7,-2.5);
            \draw [line width=2pt,color=qqwuqq] (8,-3)-- (8,-4);
            \draw [line width=2pt,color=qqwuqq] (8,-4)-- (7,-4.5);
            \draw [line width=2pt,color=qqwuqq] (7,-4.5)-- (7,-3.5);
            \draw [line width=2pt,color=qqwuqq] (7,-3.5)-- (8,-3);
            \draw [line width=2pt,color=qqwuqq] (6,-3)-- (7,-3.5);
            \draw [line width=2pt,color=qqwuqq] (7,-3.5)-- (7,-4.5);
            \draw [line width=2pt,color=qqwuqq] (7,-4.5)-- (6,-4);
            \draw [line width=2pt,color=qqwuqq] (6,-4)-- (6,-3);
            \draw [line width=2pt,color=qqwuqq] (10,-4)-- (11,-4.5);
            \draw [line width=2pt,color=qqwuqq] (11,-4.5)-- (10,-5);
            \draw [line width=2pt,color=qqwuqq] (10,-5)-- (9,-4.5);
            \draw [line width=2pt,color=qqwuqq] (9,-4.5)-- (10,-4);
            \draw [line width=2pt,color=qqwuqq] (9,-4.5)-- (9,-5.5);
            \draw [line width=2pt,color=qqwuqq] (9,-5.5)-- (10,-6);
            \draw [line width=2pt,color=qqwuqq] (10,-6)-- (10,-5);
            \draw [line width=2pt,color=qqwuqq] (10,-5)-- (9,-4.5);
            \draw [line width=2pt,color=qqwuqq] (10,-5)-- (11,-4.5);
            \draw [line width=2pt,color=qqwuqq] (11,-4.5)-- (11,-5.5);
            \draw [line width=2pt,color=qqwuqq] (11,-5.5)-- (10,-6);
            \draw [line width=2pt,color=qqwuqq] (10,-6)-- (10,-5);
            \draw (7.3,-1.2) node[anchor=north west] {\textbf{(2,1,3)}};
            \draw (9.4,-1.2) node[anchor=north west] {\textbf{(1,2,3)}};
            \draw (5.8,-3.4) node[anchor=north west] {\textbf{(3,1,2)}};
            \draw (10.8,-3.4) node[anchor=north west] {\textbf{(1,3,2)}};
            \draw (6.9,-4.918842113761193) node[anchor=north west] {\textbf{(3,2,1)}};
            \draw (9.8,-5) node[anchor=north west] {\textbf{(2,3,1)}};
            \draw [line width=2pt] (8,-3)-- (9,-3.5);
            \draw [line width=2pt] (10,-3)-- (9,-3.5);
            \draw [line width=2pt] (9,-3.5)-- (9,-4.5);
            \draw [line width=2pt] (10,-2)-- (11,-2.5);
            \draw [line width=2pt] (11,-2.5)-- (12,-2);
            \draw [line width=2pt] (12,-2)-- (11,-1.5);
            \draw [line width=2pt] (12,-2)-- (12,-3);
            \draw [line width=2pt] (11,-2.5)-- (11,-3.5);
            \draw [line width=2pt,color=zzttqq] (10,-2)-- (11,-1.5);
            \draw [line width=2pt,color=zzttqq] (11,-1.5)-- (12,-2);
            \draw [line width=2pt,color=zzttqq] (12,-2)-- (11,-2.5);
            \draw [line width=2pt,color=zzttqq] (11,-2.5)-- (10,-2);
            \draw [line width=2pt,color=zzttqq] (10,-2)-- (10,-3);
            \draw [line width=2pt,color=zzttqq] (10,-3)-- (11,-3.5);
            \draw [line width=2pt,color=zzttqq] (11,-3.5)-- (11,-2.5);
            \draw [line width=2pt,color=zzttqq] (11,-2.5)-- (10,-2);
            \draw [line width=2pt,color=zzttqq] (11,-2.5)-- (12,-2);
            \draw [line width=2pt,color=zzttqq] (12,-2)-- (12,-3);
            \draw [line width=2pt,color=zzttqq] (12,-3)-- (11,-3.5);
            \draw [line width=2pt,color=zzttqq] (11,-3.5)-- (11,-2.5);
            \draw [line width=2pt,color=zzttqq] (9,-2.5)-- (10,-3);
            \draw [line width=2pt,color=zzttqq] (10,-3)-- (9,-3.5);
            \draw [line width=2pt,color=zzttqq] (9,-3.5)-- (8,-3);
            \draw [line width=2pt,color=zzttqq] (8,-3)-- (9,-2.5);
            \draw [line width=2pt,color=zzttqq] (10,-3)-- (10,-4);
            \draw [line width=2pt,color=zzttqq] (10,-4)-- (9,-4.5);
            \draw [line width=2pt,color=zzttqq] (9,-4.5)-- (9,-3.5);
            \draw [line width=2pt,color=zzttqq] (9,-3.5)-- (10,-3);
            \draw [line width=2pt,color=zzttqq] (8,-3)-- (9,-3.5);
            \draw [line width=2pt,color=zzttqq] (9,-3.5)-- (9,-4.5);
            \draw [line width=2pt,color=zzttqq] (9,-4.5)-- (8,-4);
            \draw [line width=2pt,color=zzttqq] (8,-4)-- (8,-3);
            \draw (10.43,-1.7) node[anchor=north west] {\textbf{(1,3,3)}};
            \draw (8.4,-2.7) node[anchor=north west] {\textbf{(2,2,2)}};
            \begin{scriptsize}
                \draw [fill=ududff] (9,-0.5) circle (2.5pt);
                \draw [fill=ududff] (7,-1.5) circle (2.5pt);
                \draw [fill=ududff] (7,-2.5) circle (2.5pt);
                \draw [fill=ududff] (6,-3) circle (2.5pt);
                \draw [fill=ududff] (6,-5) circle (2.5pt);
                \draw [fill=ududff] (11,-1.5) circle (2.5pt);
                \draw [fill=ududff] (11,-2.5) circle (2.5pt);
                \draw [fill=ududff] (12,-3) circle (2.5pt);
                \draw [fill=ududff] (12,-5) circle (2.5pt);
                \draw [fill=ududff] (8,-2) circle (2.5pt);
                \draw [fill=xdxdff] (10,-1) circle (2.5pt);
                \draw [fill=xdxdff] (8,-1) circle (2.5pt);
                \draw [fill=ududff] (10,-2) circle (2.5pt);
                \draw [fill=ududff] (10,-4) circle (2.5pt);
                \draw [fill=ududff] (8,-4) circle (2.5pt);
                \draw [fill=ududff] (7,-3.5) circle (2.5pt);
                \draw [fill=ududff] (9,-2.5) circle (2.5pt);
                \draw [fill=ududff] (11,-3.5) circle (2.5pt);
                \draw [fill=ududff] (9,-3.5) circle (2.5pt);
                \draw [fill=ududff] (7,-4.5) circle (2.5pt);
                \draw [fill=ududff] (11,-4.5) circle (2.5pt);
                \draw [fill=xdxdff] (6,-4) circle (2.5pt);
                \draw [fill=ududff] (8,-5) circle (2.5pt);
                \draw [fill=ududff] (8,-6) circle (2.5pt);
                \draw [fill=ududff] (11,-5.5) circle (2.5pt);
                \draw [fill=ududff] (10,-6) circle (2.5pt);
                \draw [fill=ududff] (10,-5) circle (2.5pt);
                \draw [fill=ududff] (9,-5.5) circle (2.5pt);
                \draw [fill=ududff] (9,-4.5) circle (2.5pt);
                \draw [fill=xdxdff] (7,-5.5) circle (2.5pt);
                \draw [fill=uuuuuu] (8,-3) circle (2pt);
                \draw [fill=xdxdff] (9,-1.5) circle (2.5pt);
                \draw [fill=xdxdff] (10,-3) circle (2.5pt);
                \draw [fill=xdxdff] (12,-4) circle (2.5pt);
                \draw [fill=ududff] (12,-2) circle (2.5pt);
            \end{scriptsize}
        \end{tikzpicture} 
        \caption{The closure diagram $\overline{\calD}$}
        \label{Shadow} 
    \end{center}
\end{figure}
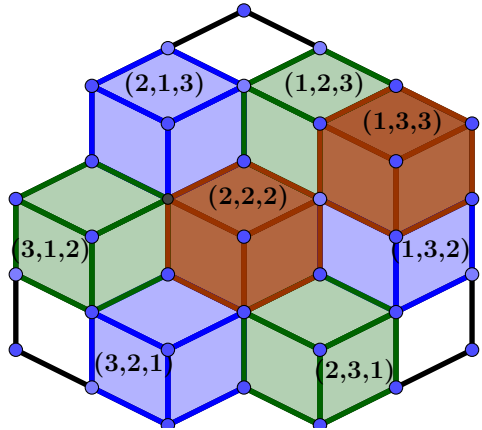

Here, we explain the introduction of the projection property.

\begin{Remark}
    \label{why-shadow}
    Let $\calD$ be a three-dimensional Ferrers diagram. If $a_{\calD}=1$ or $b_{\calD}=1$ or $c_{\calD}=1$, then $\calD$ is essentially a two-dimensional Ferrers diagram. In this case, $\calD$ automatically satisfies the projection property. From this point of view, three-dimensional Ferrers diagrams which satisfy the projection property are more natural as generalizations of two-dimensional Ferrers diagrams to the three-dimensional case. As a matter of fact, the essential reason that we introduce the projection property lies in the desire to achieve Koszul property. As a necessary condition, the degree 3 generator $T_{1,2,3}T_{2,3,1}T_{3,1,2}-T_{1,3,2}T_{2,1,3}T_{3,2,1}$ in \Cref{MinExam} is not expected to exist in any minimal generating set, due to \cite[Proposition 6.3]{MR2850142}. To remove it with respect to a bigger three-dimensional diagram, one has to seek support from the ubiquitous quadratic generators in Definition \ref{2-minors} via Lemma \ref{Deg2Coincide}.  With this in mind, and if we stick to three-dimensional Ferrers diagrams, $(1,3,3)$ will be the optimal element to attach to the diagram $\calD$ of \Cref{MinExam}. After this is done, the degree 3 generator of the toric ideal will soon be reduced by combinations of quadratic generators.  Notice that in this maneuver, the position $(1,3,3)$ of  $x=1$ layer lies in the projection of the $x=2$ layer, which only respects two adjacent layers. To make it more induction-friendly, we naturally come up with the projection property in Definition \ref{shadow}. 
\end{Remark}

\begin{Observation}
    \label{truncation-2}
    Suppose that $\calD$ is a three-dimensional Ferrers diagram satisfying the
    projection property. Then all the three truncated subdiagrams 
    \[
        \Set{(i,j,k)\in \calD: i\ne i_0}, \quad
        \Set{(i,j,k)\in \calD: j\ne j_0} \text{ and }
        \Set{(i,j,k)\in \calD: k\ne k_0}
    \]
    are essentially three-dimensional Ferrers diagrams which still satisfy the
    projection property.
\end{Observation}

Return to $S=\KK[\bfT_{\calD}]=\KK[T_{i,j,k}: (i,j,k)\in \calD]$. Throughout
this paper, the variables will be ordered such that $T_{i,j,k}>T_{i',j',k'}$ if
and only if the tuple $(i,j,k)$ precedes $(i',j',k')$ lexicographically. With
respect to this order of variables, we will consider the lexicographic monomial
order $\prec_{lex}$ on the set of monomials in $S$. A binomial ideal is called
\emph{lexicographically quadratic} if all its minimal Gr\"obner basis elements
with respect to the lexicographic order are quadratic.  When the special fiber
ideal $J_{\calD}$ corresponding to $\calD$ is lexicographically quadratic, we
will simply say that $\calD$ is \emph{lexicographically quadratic}.

For a monomial ideal $I$, we write $\gens(I)$ for the set of its minimal monomial generators. Meanwhile, if $J$ is a binomial ideal, $\ini(J)$ is the initial ideal with respect to the lexicographic monomial order. This is also a monomial ideal. 

In the meantime, given a three-dimensional diagram $\calD$ and a point $\bdu\in \calD$, we always let $\calD_{\bdu}$ be the diagram obtained from $\calD$ by removing those points preceding $\bdu$ lexicographically. This notation benifits our induction argument later in this paper.

The following property shows that the generating set and its initial part can
be inherited by suitable subdiagrams.

\begin{Proposition}
    \label{truncation-formula} 
    Suppose that $\calD$ is a finite three-dimensional diagram and $\bdu=(i_0,j_0,k_0)\in \calD$. Let $\calG$ be one of the following three truncated subdiagrams 
    \[
        \Set{(i,j,k)\in \calD: i\ne i_0}, \quad
        \Set{(i,j,k)\in \calD: j\ne j_0} \text{ and }
        \Set{(i,j,k)\in \calD: k\ne k_0}
    \]
    or $\calG=\calD_{\bdu}$ as defined above. Then the following restriction formulas hold:
    \[
        J_{\calD}\cap \KK[\bfT_{\calG}]=J_{\calG} \quad \text{ and } \quad
        \gens(\ini(J_{\calD}))\cap \KK[\bfT_{\calG}]=\gens(\ini(J_{\calG})). 
    \]
\end{Proposition}

\begin{proof}
    \begin{enumerate}[a]
        \item Without loss of generality, we consider the subdiagram
            \[
                \calG=\Set{(i,j,k)\in \calD: i\ne i_0}.
            \]
            It is clear that $J_{\calG}\subseteq J_{\calD}$. On the other hand,
            take arbitrary binomial 
            \begin{equation}
                \label{f}
                f=T_{\bdu_1}T_{\bdu_2}\cdots T_{\bdu_n}-T_{\bdv_1}T_{\bdv_2}\cdots
                T_{\bdv_n} 
            \end{equation}
            in the binomial ideal $J_{\calD}$. It follows from the definition
            of $J_{\calD}$ that one of the lattice points $\bdu_1,\dots,\bdu_n$
            is on the $x=i_0$ layer if and only if one of the lattice points
            $\bdv_1,\dots,\bdv_n$ is on the $x=i_0$ layer. This implies that
            \[
                J_{\calD}\cap \KK[\bfT_{\calG}]=J_{\calG}.
            \] 
            When $f\in J_{\calG}$, it is also clear that
            $T_{\bdu_1}T_{\bdu_2}\cdots T_{\bdu_n}$ is the leading term of $f$
            in $\KK[\bfT_{\calG}]$ if and only if it is so in
            $\KK[\bfT_{\calD}]$. This implies that
            \[
                \ini(J_{\calD})\cap \KK[\bfT_{\calG}]\supseteq\ini(J_{\calG}).  
            \]
            On the other hand, take arbitrary monomial $T_{\bdu_1} T_{\bdu_2}
            \cdots T_{\bdu_n}\in \ini(J_{\calD})\cap \KK[\bfT_{\calG}]$. By the
            definition of the initial ideal, $T_{\bdu_1}T_{\bdu_2}\cdots
            T_{\bdu_n}$ is the leading term of some
            \[
                f=T_{\bdu_1}T_{\bdu_2}\cdots T_{\bdu_n}-T_{\bdv_1}T_{\bdv_2}\cdots
                T_{\bdv_n}\in J_{\calD}.
            \]
            As argued above, since $T_{\bdu_1}T_{\bdu_2}\cdots T_{\bdu_n}\in
            \KK[\bfT_{\calG}]$, we also have $f\in \KK[\bfT_{\calG}]$. This
            means that $f\in J_{\calD}\cap \KK[\bfT_{\calG}]=J_{\calG}$. And
            the leading term $T_{\bdu_1}T_{\bdu_2}\cdots T_{\bdu_n}\in
            \ini(J_{\calG})$. Therefore, we have shown that
            \[
                \ini(J_{\calD})\cap \KK[\bfT_{\calG}]=\ini(J_{\calG}).  
            \]    
            Consequently,
            \
            \[
                \gens(\ini(J_{\calD}))\cap \KK[\bfT_{\calG}]=\gens(\ini(J_{\calG})). 
            \]
        \item Now, consider the case when $\calG=\calD_{\bdu}$. One simply notice that, using the notation in \eqref{f}, since $T_{\bdu_1} T_{\bdu_2} \cdots T_{\bdu_n}$ is the leading term of $f$ with respect to the lexicographic order, $T_{\bdu_1} T_{\bdu_2} \cdots T_{\bdu_n}\in \KK[\bfT_{\calG}]$ if and only if $f\in \KK[\bfT_{\calG}]$.  The remaining proof is similar to the previous case. 
            \qedhere
    \end{enumerate}
\end{proof}

\begin{Corollary}
    \label{truncation}
    Suppose that $\calD$ is a finite three-dimensional diagram.  If $\calD$ is
    lexicographically quadratic and $\bdu=(i_0,j_0,k_0)\in \calD$, then all the
    three truncated subdiagrams 
    \[
        \Set{(i,j,k)\in \calD: i\ne i_0}, \quad
        \Set{(i,j,k)\in \calD: j\ne j_0} \text{ and }
        \Set{(i,j,k)\in \calD: k\ne k_0}
    \]
    as well as $\calD_{\bdu}$ are again lexicographically quadratic.
\end{Corollary}

We are now ready to introduce the main subject discussed in this paper.

\begin{Definition}
    \label{2-minors}
    Let $\calD$ be a diagram of finite lattice points in $\ZZ_+^3$.  For
    $\bdu=(i_1,j_1,k_1)$ and $\bdv=(i_2,j_2,k_2)$ in $\calD$, define
    \[
        \bfI_{2,x}(\bdu,\bdv)\coloneqq 
        \begin{cases}
            T_{\bdu}T_{\bdv}-T_{i_2,j_1,k_1}T_{i_1,j_2,k_2}, & \text{if
                $(i_2,j_1,k_1), (i_1,j_2,k_2)\in \calD$},\\
            0, & \text{otherwise}.
        \end{cases}
    \]
    We will simply say \emph{switching the $x$-coordinates} in the first case.
    We can similarly define $\bfI_{2,y}$ and $\bfI_{2,z}$. Now, let
    \[
        \bfI_2(\calD)\coloneqq \Big(\bfI_{2,x}(\bdu,\bdv),\bfI_{2,y}(\bdu,\bdv),
        \bfI_{2,z}(\bdu,\bdv): \bdu, \bdv\in \calD\Big)\subset
        \KK[\bfT_{\calD}],
    \]
    and call it the \emph{2-minors ideal} of $\calD$.  
\end{Definition}

Obviously, when $\calD$ is essentially a two-dimensional diagram,
$\bfI_2(\calD)$ is the traditional 2-minors ideal of $\calD$. 

In the following, we will investigate $\bfI_2(\calD)$ and
$\bfI_2(\calD_{\bdu})$.  It is clear that $\bfI_2(\calD)\subseteq J_{\calD}$,
the special fiber ideal corresponding to the diagram $\calD$. Notice that the
choice of $\bfI_{2,x}(\bdu,\bdv)$, $\bfI_{2,y}(\bdu,\bdv)$ and
$\bfI_{2,z}(\bdu,\bdv)$ is not by accident. Those elements are actually the
degree two binomials of $J_{\calD}$.

\begin{Lemma}
    \label{Deg2Coincide}
    If the nonzero binomial $f=T_{\bdu}T_{\bdv}-T_{\bdu'}T_{\bdv'}$ belongs to
    $J_{\calD}$, then it is one of $\bfI_{2,x}(\bdu,\bdv)$,
    $\bfI_{2,y}(\bdu,\bdv)$ and $\bfI_{2,z}(\bdu,\bdv)$. 
\end{Lemma}

\begin{proof}
    We may write $\bdu=(i_1,j_1,k_1),
    \bdv=(i_2,j_2,k_2),\bdu'=(i_1',j_1',k_1')$ and $\bdv'=(i_2',j_2',k_2')$. As
    multi-sets, we have
    \[
        \Set{i_1,i_2}=\Set{i_1',i_2'},\quad 
        \Set{j_1,j_2}=\Set{j_1',j_2'}\quad \text{and} \quad
        \Set{k_1,k_2}=\Set{k_1',k_2'}.
    \] 
    Without loss of generality, we may assume that $i_1'=i_1$ and $i_2'=i_2$.
    We have the following three cases.
    \begin{enumerate}[a]
        \item If $j_1'=j_1$, then $k_1'=k_2$ and $f=\bfI_{2,z}(\bdu,\bdv)$.
        \item If $k_1'=k_1$, then $j_1'=j_2$ and $f=\bfI_{2,y}(\bdu,\bdv)$.
        \item If $j_1'=j_2$ and $k_1'=k_2$, then $f=\bfI_{2,x}(\bdu,\bdv)$. \qedhere
    \end{enumerate} 
\end{proof}

\begin{Remark}
    The amiable fact that the degree 2 generators only appear in the form of 2-minors cannot be directly generalized to four-dimensional case, as manifested from the above proof.
\end{Remark}

The main goal of this section is to show the quadratic binomials in $\bfI_2(\calD)$, defined in \Cref{2-minors}, form a Gr\"obner basis of this ideal with respect to the lexicographic order. The next proposition observes that if the toric ring is defined by a three-dimensional Ferrers diagram which satisfies the projection property, then the toric ideal $J_{\calD}$ cannot have any degree three element in the minimal Gr\"obner basis.  Since the degree two part of $J_{\calD}$ coincides with that of $\bfI_2(\calD)$, this leads to the lexicographically quadratic property of the latter ideal.

\begin{Proposition}
    \label{no3}
    Let $\calD$ be a three-dimensional Ferrers diagram which satisfies the projection property. Then none of the minimal Gr\"obner basis element with respect to lexicographic order of the special fiber ideal has degree three.   
\end{Proposition}

\begin{proof}
    Notice that all the points involved in such a potential binomial is contained in a $3\times3\times 3$ cube. But the Ferrers diagram property and the projection property are all preserved under layer truncations. Thus, by \Cref{truncation-formula}, it suffices to show that for any three-dimensional Ferrers diagram governed by the point $(3,3,3)$, if it satisfies the projection property, then it is lexicographically quadratic. For this, we can verify by running Macaulay2 \cite{M2} and exhaust all possible cases. One can check, for instance, by running the \texttt{All3()} in the script \texttt{Ferrers3D.m2}. The latter script is attached to the arXiv version (arXiv:1709.03251) of this work, and is also accessible at \texttt{http://www.personal.psu.edu/kul20/Ferrers3D.m2}.
\end{proof}

\begin{Corollary}
    \label{I2lq}
    Let $\calD$ be a three-dimensional Ferrers diagram which satisfies the projection property or is lexicographically quadratic. Then the $2$-minors ideal $\bfI_2(\calD)$ is lexicographically quadratic.
\end{Corollary}

\begin{proof}
    Since $\bfI_2(\calD)\subseteq J_{\calD}$ and their degree two parts agree by \Cref{Deg2Coincide}, this result follows from \Cref{no3} and the well-known Buchberger's criterion \cite[Theorem 7.3]{MR2363237}.
\end{proof}

\begin{Remark}
    \label{why-shadow-grobner}
    When dealing with Gr\"obner basis by using Buchberger's criterion, one needs to show those $S$-pairs can be reduced to $0$. The projection property provides the sufficient condition for this purpose when all the $a_{\calD}$, $b_{\calD}$ and $c_{\calD}$ are at least three, in view of \Cref{no3}. Therefore, $\calD$ in our mind is a relatively large and more general diagram.  Of course, one can construct diagrams with very few elements, not satisfying the projection property condition,  but still have quadratic Gr\"obner basis.  See, for instance, the subsequent example.  It implies that the projection property is not a necessary condition. It also demonstrates that the Gr\"obner basis is not necessarily quadratic if the monomial ordering is not lexicographic.
\end{Remark}

\begin{Example}
    Consider the three-dimensional Ferrers diagram, generated by the lattice points
    \[
        (1,3,2),(2,1,3), (2,3,1), (3,1,2), (3,2,1). 
    \]
    Since the lattice point $(2,2,2)$ is missing, this diagram does not satisfy the projection property. However, using \texttt{Macaulay2} \cite{M2}, we can verify that this diagram is lexicographically quadratic. Furthermore, if we change the monomial ordering from the lexicographic order to the graded reverse lexicographic order, then the Gr\"obner basis of the presentation ideal is no longer quadratic. Indeed, non-squarefree binomials of degree three emerge.
\end{Example}

With a minor restriction, the lexicographically quadratic property can be obtained in a more general setting. 

\begin{Proposition}
    \label{I2-inheritance} 
    Let $\calD$ be a three-dimensional diagram such that $\bfI_2(\calD)$ is lexicographically quadratic. Let $\calG\subset \calD$ be a subdiagram satisfying one of the following conditions.
    \begin{enumerate}[labelindent=\parindent, leftmargin=*, align=left]
        \item[\textup{(Detaching Condition)}] For any $\bdv_1,\bdv_2\in \calG$
            such that $T_{\bdv_1}T_{\bdv_2}-T_{\bdv_1'}T_{\bdv_2'}\in
            \bfI_2(\calD)$, we will always have $\bdv_1',\bdv_2'\in\calG$.
        \item[\textup{(Leading Monomial Condition)}] $\calG=\calD\setminus
            \bdu$ and for each $\bdv\in \calG$ such that
            $T_{\bdu}T_{\bdv}-T_{\bdu'}T_{\bdv'}\in \bfI_2(\calD)$, the
            monomial $T_{\bdu}T_{\bdv}$ is the leading monomial of this
            binomial. 
    \end{enumerate}
    Then the following restriction formulas hold:
    \[
        \bfI_2(\calD)\cap \KK[\bfT_{\calG}]=\bfI_2(\calG) \quad \text{ and
        }\quad  \gens(\ini(\bfI_2(\calD)))\cap
        \KK[\bfT_{\calG}]=\gens(\ini(\bfI_2(\calG))). 
    \]
    In particular, $\bfI_2(\calG)$ is lexicographically quadratic.
\end{Proposition}

\begin{proof} 
    For the first equality, it is clear that $\bfI_2(\calD)\cap
    \KK[\bfT_{\calG}]\supseteq \bfI_2(\calG)$. Thus, it suffices to show the
    reverse containment $\bfI_2(\calD)\cap \KK[\bfT_{\calG}]\subseteq
    \bfI_2(\calG)$.  For this, we take arbitrary binomial $f\in
    \bfI_2(\calD)\cap \KK[\bfT_{\calG}]$. Suppose for contradiction that
    $f\notin \bfI_2(\calG)$.  We can replace $f$ by its remainder with respect
    to $\ini(\bfI_2(\calG))$.  Hence, none of the terms of $f$ belongs to
    $\ini(\bfI_2(\calG))$.  Now we may assume that $f=f_1-f_2$ with $f_1$ being
    the leading monomial and $f_1\notin \ini(\bfI_2(\calG))$. Notice that
    $f_1,f_2\in \KK[\bfT_{\calG}]$. Since $f\in \bfI_2(\calD)$ while
    $\bfI_2(\calD)$ is lexicographically quadratic, we can find some quadratic
    binomial $g=g_1-g_2\in \bfI_2(\calD)$ with $g_1$ being the leading monomial
    and $g_1$ being a factor of $f_1$. Since $f_1\in \KK[\bfT_{\calG}]$, we
    will have $g_1 \in \KK[\bfT_{\calG}]$ as well.  Under the detaching
    condition, we directly get $g_2\in \KK[\bfT_{\calG}]$. Under the leading
    monomial condition, since $g_1$ is the leading monomial and $T_{\bdu}$ does
    not divide $g_1$, $T_{\bdu}$ does not divide $g_2$ as well. This also means
    that $g_2\in \KK[\bfT_{\calG}]$.  In turn, we have $g\in \KK[\bfT_{\calG}]$
    and consequently $g\in \bfI_2(\calG)$.  This implies that $f_1\in
    \ini(\bfI_2(\calG))$, a contradiction.

    For the second equality, it suffices to show that
    $\gens(\ini(\bfI_2(\calD))) \cap \KK[\bfT_{\calG}]\subseteq
    \gens(\ini(\bfI_2(\calG)))$.  Take arbitrary quadratic monomial $g_1\in
    \ini(\bfI_2(\calD))\cap \KK[\bfT_{\calG}]$. By the definition, we can find
    a binomial $g=g_1-g_2\in \bfI_2(\calD)$ with $g_1$ being the leading
    monomial.  Now, the arguments in the previous paragraph still shows that
    $g\in \KK[\bfT_{\calG}]$ and $g_1\in \ini(\bfI_2(\calG))$. 
\end{proof}

\begin{Remark}
    \label{I2-inheritance-application} 
    \begin{enumerate}[a]
        \item The detaching condition is satisfied when $\bdu=(i_0,j_0,k_0)\in
            \calD$ and $\calG$ is one of the following three truncated
            subdiagrams 
            \[
            \Set{(i,j,k)\in \calD: i\ne i_0}, \quad
            \Set{(i,j,k)\in \calD: j\ne j_0} \text{ and }
            \Set{(i,j,k)\in \calD: k\ne k_0}.
            \] 
        \item The leading monomial condition is automatically satisfied when
            $\bdu$ is lexicographically the first point of $\calD$.
    \end{enumerate}
\end{Remark}

\section{Simplicial complex of the initial ideal}
\label{section:Cohen--Macaulay}

Let $\calD$ be a three-dimensional Ferrers diagram which satisfies the projection property. Notice that
the initial ideal of $\bfI_2(\calD)$ is squarefree by \Cref{I2lq}.  The Stanley--Reisner
complex of this initial ideal will be denoted by $\Delta(\calD)$. To be more
specific, 
\[
    \Delta(\calD)\coloneqq \Set{F\subseteq \{T_{\bdu}:\bdu\in \calD\}:
        \prod_{T_{\bdv}\in F}T_{\bdv}\notin \ini(\bfI_2(\calD))}.
\]
 For a subdiagram $\calG\subseteq \calD$, we use
$\Delta(\calD,\calG)$ to represent the restriction complex of $\Delta(\calD)$
to the set $\Set{T_{\bdu}: \bdu\in \calG}$.
On the other hand, for a given simplicial complex $\Delta$ over some set
$\bdX$, let $I_{\Delta}$ be the Stanley--Reisner ideal in the corresponding
polynomial ring $\KK[\bdX]$. 

The purpose of the next section is to show that when $\calD$ is a
three-dimensional Ferrers diagram which satisfies the projection property, the ideal
$\bfI_2(\calD)$ is Cohen--Macaulay. By \cite[Corollary 3.3.5]{MR2724673}, it
suffices to show that $\ini(\bfI_2(\calD))$ is Cohen--Macaulay. To achieve
this, we prove in \Cref{VD-InductionOrder} that the Stanley--Reisner complex,
$\Delta(\calD)$, is pure vertex-decomposable; see \Cref{VD-Definition}. It is
well-known that pure vertex-decomposable complexes are pure shellable, hence
Cohen--Macaulay, or equivalently, their Stanley--Reisner ideals are
Cohen--Macaulay.

In this section, we will recall and build additional tools for the proofs in
the sequel. In particular, we need to determine the dimensions of
the restriction complexes that are involved in the those proofs.

\begin{Remark}
    Throughout this paper, we use implicitly the following well-known fact.
    Let $T_{\bdu}$ be a vertex of a simplicial complex $\Delta$ on $\calV$.
    Then the cone over the link complex $\link_{\Delta}(T_{\bdu})$ with apex
    $T_{\bdu}$ considered as a complex on $\calV$ has Stanley--Reisner ideal
    $I_{\Delta}:T_{\bdu}$, and the Stanley--Reisner ideal of the deletion
    complex $\Delta\setminus T_{\bdu}$ considered as a complex on $\calV$ is
    $(T_{\bdu},I_{\Delta\setminus T_{\bdu}}')$, where $I_{\Delta\setminus
        T_{\bdu}}'$ is the Stanley--Reisner ideal of $\Delta\setminus T_{\bdu}$
    considered as a complex on $\calV\setminus T_{\bdu}$. 
\end{Remark}

\begin{Definition}
    [{\cite{MR0593648}}] 
    \label{VD-Definition}
    A pure simplicial complex $\Delta$ is said to be \emph{vertex-decomposable} if $\Delta$ is a simplex or equal to $\Set{\varnothing}$, or there exists a vertex $v$ such that the link complex $\lk_{\Delta}(v)$ and the deletion complex $\Delta\setminus v$ are both pure and vertex-decomposable and $\dim(\Delta) = \dim(\Delta\setminus v) = \dim(\lk_{\Delta}(v)) + 1$. The vertex $v$ here is called a \emph{shedding vertex}.  
\end{Definition}

\begin{Remark}
    \label{VD-cone}
    Suppose that $\Delta$ is a pure simplicial complex and a cone with apex $v$. It follows from \cite[Proposition 2.4]{MR0593648} that $\Delta$ is vertex-decomposable if and only if $\Delta\setminus v$ is so.
\end{Remark}

For a general finite diagram $\calD$ in $\ZZ_{+}^3$, we use the superscript to denote the corresponding $x$ layers. For instance,
\[
\calD^1\coloneqq \Set{(1,j,k)\in \calD} \quad \text{ and } \quad
\calD^{\ge a}\coloneqq \Set{(i,j,k)\in \calD: i\ge a}.
\] 

\begin{Definition}
Let $\prec$ be a total order on $\calD$. We say that $\prec$ is a
\emph{quasi-lexicographic order} if it satisfies the following two conditions.
\begin{enumerate}[labelindent=\parindent, leftmargin=*, align=left]
    \item [{\textup(QLO-1)}] The points in $\calD^1$ precede the points in
        $\calD^{\ge 2}$ with respect to $\prec$.
    \item [{\textup(QLO-2)}] For distinct $\bdu=(1,j_1,k_1)$ and
        $\bdv=(1,j_2,k_2)$ in $\calD$, if $j_1\le j_2$ and $k_1\le k_2$, then
        $\bdu$ precedes $\bdv$ with respect to $\prec$.
\end{enumerate}
\end{Definition}

Obviously, the lexicographic order is a quasi-lexicographic order. Throughout
this section, we always assume that $\prec$ is a quasi-lexicographic order.

Given a lattice point $\bdu\in \calD^1$, let $\calA_{\bdu}$ be the diagram
obtained from $\calD$ by removing the points before $\bdu$ with respect to
$\prec$. We also write $\calA_{\bdu}^{+}\coloneqq \calA_{\bdu}\setminus \bdu$. Notice that when our quasi-lexicographic order $\prec$ happens to be the lexicographic order, then $\calA_{\bdu}=\calD_{\bdu}$.

\begin{Remark}
    \label{restriction-QLO}
    Let $\calD$ be a finite three-dimensional diagram such that $\bfI_2(\calD)$ is lexicographically quadratic. Suppose that $\bdu=(1,j_1,k_1)\in \calD^1$ is the first point with respect to $\prec$ and let $\calG=\calD\setminus \bdu=\calA_{\bdu}^{+}$.  By (QLO-2), those $\bdv=(1,j_2,k_2)$ preceding $\bdu$ lexicographically must satisfy $j_2<j_1$ and $k_2>k_1$.  In particular, $(1,j_2,k_1)\notin \calD$. This implies that the ``leading monomial condition'' in \Cref{I2-inheritance} holds. Thus, we have the following restriction formulas:
    \[
        \bfI_2(\calD)\cap \KK[\bfT_{\calG}]=\bfI_2(\calG) \quad \text{ and
        }\quad  \gens(\ini(\bfI_2(\calD)))\cap
        \KK[\bfT_{\calG}]=\gens(\ini(\bfI_2(\calG))). 
    \]
    By induction, for any $\bdw\in \calD^{1}$, we have similar formulas:
      \[
        \bfI_2(\calD)\cap \KK[\bfT_{\calA_{\bdw}}]=\bfI_2(\calA_{\bdw}) \quad \text{ and
        }\quad  \gens(\ini(\bfI_2(\calD)))\cap
        \KK[\bfT_{\calA_{\bdw}}]=\gens(\ini(\bfI_2(\calA_{\bdw}))). 
    \] 
    In particular, the restriction complexes $\Delta(\calD,\calA_{\bdw})=\Delta(\calA_{\bdw})$.
\end{Remark}

With respect to the diagram $\calD$ above, define
\begin{align*}
    N(\calD)\coloneqq \Set{\bdu\in \calD^{1}: \ini(\bfI_2(\calA_{\bdu}))\supsetneq
        \ini(\bfI_2(\calA_{\bdu}^{+}))\KK[\bfT_{\calA_{\bdu}}]}
\end{align*}
and $\Phan(\calD)\coloneqq  \calD^{1}\setminus N(\calD)$ to be the set of \emph{normal
    points} and \emph{phantom points} with respect to the quasi-lexicographic
order $\prec$ respectively.  

\begin{Remark}
    \label{equivalent-definition}
    Let $\calD$ be a finite three-dimensional diagram such that $\bfI_2(\calD)$
    is lexicographically quadratic. For each $\bdu\in \calD^1$, as pointed out
    in \Cref{restriction-QLO}, the pair $\calA_{\bdu}^{+}\subset \calA_{\bdu}$
    satisfies the leading monomial condition. Thus, $\bdu$ is a normal point if
    and only if $\bdu$ can switch coordinates with some $\bdv\in
    \calA_{\bdu}^{+}$.  Dually, $\bdu$ is a phantom point if and only if
    $\Delta(\calA_{\bdu})$ is a cone over $\Delta(\calA_{\bdu})\setminus
    T_{\bdu}=\Delta(\calA_{\bdu}^{+})$ with the apex $T_{\bdu}$.
\end{Remark}

Indeed, when $\calD$ is a three-dimensional Ferrers diagram satisfying the
projection property, the sets of normal points and phantom points are independent
of the specific quasi-lexicographic order that we choose.

\begin{Discussion}
    \label{phantom-x1}
    Let $\calD$ be a three-dimensional Ferrers diagram satisfying the
    projection property and $a_{\calD}\ge 2$.  It is clear that if
    $\bdu=(1,j_1,k_1)\in\calD$ is a phantom point, then it is a border point,
    i.e., $(1,j_1+1,k_1+1)\notin \calD$.  Indeed, if otherwise, then $\bdu$
    can make a $y$-coordinates switch with the point $(1,j_1+1,k_1+1)$.
    
    Denote the set of such border points by $\calB$.  Now, we apply
    \Cref{equivalent-definition} to exclude normal points in $\calB$.  The
    following border points are normal points.
    \begin{enumerate}[i]
        \item [\text{$(\calB_y)$}] The border points on the
            $y=1,{2},\dots,b_{\calD^{\ge 2}}-1$ lines with minimal
            $z$-coordinates.
        \item [\text{$(\calB_z)$}] The border points on the
            $z=1,{2},\dots,c_{\calD^{\ge 2}}-1$ lines with minimal
            $y$-coordinates.
    \end{enumerate}
    For $\bdu=(1,j_1,k_1)\in \calB_y$, we have $(1,j_1+1,k_1)\in \calD$ and
    $\bdu$ can switch $y$-coordinates with $(2,j_1+1,1)$. On the other hand,
    if $\bdu$ is a border point which can switch $y$-coordinates within
    $\calA_{\bdu}$, it can only do so with points in $\calD^{\ge 2}$ by (QLO-2).
    Hence such $\bdu\in \calB_y$. We can similarly talk about $\calB_z$.
    Thus, border points in these two sets are the border points such that
    $y$-coordinates switch or $z$-coordinates switch is available.
    Furthermore, these two sets are disjoint. To see this, suppose that
    $\bdu=(1,j_1,k_1)\in \calB_y\cap\calB_z$. Since $\bdu\in \calB_y$,
    $(2,j_1+1,1)\in \calD$. Similarly, $(2,1,k_1+1)\in \calD$. By the
    projection property, this leads to $(1,j_1+1,k_1+1)\in \calD$. But $\bdu$
    is supposed to be a border point. This is a contradiction.
    
    For $\bdv=(1,j_2,k_2)\in \calB\setminus(\calB_y\cup\calB_z)$, it is a
    normal point if and only if it can switch $x$-coordinates with some $\bdv$
    later than $\bdv$ with respect to $\prec$. A necessary condition for this
    to happen is $(2,j_2,k_2)\in \calD$. Hence $b_{\calD^{\ge 2}}\ge j_2$
    and $c_{\calD^{\ge 2}}\ge k_2$. By the projection property, $(1,b_{\calD^{\ge
    2}},c_{\calD^{\ge 2}})\in \calD$. Since $\bdv$ is a border point, either
    $j_2=b_{\calD^{\ge 2}}$ or $k_2=c_{\calD^{\ge 2}}$. If $j_2=b_{\calD^{\ge
    2}}$ while $k_2<c_{\calD^{\ge 2}}$, then $z$-coordinates switch is
    available and $\bdv\in \calB_z$. We can similarly discuss the symmetric
    case. Since we assume $\bdv\notin \calB_y\cup \calB_z$ in this case, we
    must have $\bdv=(1,b_{\calD^{\ge 2}},c_{\calD^{\ge 2}})$. By (QLO-2), this
    is the only point in $\calA_{\bdv}^1$ which can increase its
    $x$-coordinate. Thus, $x$-coordinates switch is not possible for
    $\bdv$.

    Therefore, the normal points and phantom points with respect to $\prec$
    agree with those defined with respect to the lexicographic order.
    Furthermore, the above discussion shows that 
    \[
    \left| N(\calD) \cap \calB \right|= b_{\calD^{\ge 2}}+c_{\calD^{\ge 2}}-2.
    \]
    Since $\left| \calB \right|=b_{\calD}+c_{\calD}-1$,
    \begin{equation}
        \label{x1-phantonpoints-number}
        \left| \Phan(\calD) \right|= \left| \calB \right|-\left| N(\calD) \cap
            \calB \right|=(b_{\calD}-b_{\calD^{\ge
                2}})+(c_{\calD}-c_{\calD^{\ge 2}})+1.
    \end{equation}
    This is the expected number in view of \Cref{DimI2} and
    \Cref{VD-InductionOrder}.
\end{Discussion}

\begin{Observation}
    \label{HeightFormula}
    Let $\calD$ be a finite three-dimensional Ferrers diagram such that
    $\bfI_2(\calD)$ is lexicographically quadratic. Take arbitrary $\bdu\in
    \calD^1$. Suppose that $\ini(\bfI_2(\calA_{\bdu}))$ and
    $\ini(\bfI_2(\calA_{\bdu}^{+}))$ are unmixed, or, equivalently,
    $\Delta(\calA_{\bdu})$ and $\Delta(\calA_{\bdu}^{+})$ are pure.
    \begin{enumerate}[i]
        \item If $\bdu$ is a phantom point, then trivially
            $\codim\bfI_2(\calA_{\bdu})=\codim\bfI_2(\calA_{\bdu}^{+})$. 
        \item If $\bdu$ is a normal point, then $\Delta(\calA_{\bdu})$ and
            $\Delta(\calA_{\bdu}^{+})$ have the same dimension by the
            definition of $N(\calD)$, \Cref{restriction-QLO} and
            \Cref{equivalent-definition}. Hence,
            $\codim\bfI_2(\calA_{\bdu})=\codim\bfI_2(\calA_{\bdu}^{+})+1$.
    \end{enumerate}
\end{Observation}

\begin{Lemma}
    \label{DimI2}
    If $\calD$ is a three-dimensional Ferrers diagram with $\Delta(\calD)$
    being pure, then the dimension of $\Delta(\calD)$ is
    $a_{\calD}+b_{\calD}+c_{\calD}-3$.
\end{Lemma}

\begin{proof}
    One checks with ease that 
    \begin{align*}
        \{ T_{1,1,1},T_{2,1,1},\dots,T_{a_{\calD},1,1}, T_{1,2,1},T_{1,3,1},
        \dots,T_{1,b_{\calD},1}, T_{1,1,2},T_{1,1,3},\dots,T_{1,1,c_{\calD}}\}
    \end{align*}
    forms a facet of $\Delta(\calD)$. And its cardinality is exactly
    $a_{\calD}+b_{\calD}+c_{\calD}-2$.
\end{proof}

\begin{Lemma}
    \label{DimFiber}
    Let $\calD$ be a three-dimensional Ferrers diagram. Then
    $\dim(\calF(\calD))=a_{\calD}+b_{\calD}+c_{\calD}-2$.
\end{Lemma}

\begin{proof}
    One can check that the images of 
    \begin{align*}
        \{ T_{1,1,1},T_{2,1,1},\dots,T_{a_{\calD},1,1},T_{1,2,1},T_{1,3,1},
        \dots,T_{1,b_{\calD},1},T_{1,1,2},T_{1,1,3},\dots,T_{1,1,c_{\calD}}\}
    \end{align*}
form a transcendental basis of the domain $\calF(\calD)$ over $\KK$. Now, we
may apply \cite[Theorem A.16]{MR1251956}.
\end{proof}


Now, at the end of this section, we advertise the attacking tactics for
\Cref{VD-InductionOrder} and \Cref{I2prime}. Let $\calD$ be a three-dimensional
Ferrers diagram.  In the proof of those results, we need to decompose the
diagram $\calD$ in a rather involved way. Because of this, take arbitrary
$\bdu=(1,j_0,k_0)\in \calD$ with 
\[
    \alpha=a_{\calD}(\bdu),\quad 
    \beta=b_{\calD}(\bdu)\quad \text{and} \quad
    \gamma=c_{\calD}(\bdu)
\]
as defined before \Cref{3F}, we divide $\calD$ into the following six zones:
\begin{align*}
    \calZ_1(\calD,\bdu)\coloneqq &\Set{(i,j,k)\in \calD : 1\le j\le j_0 \text{ and } k>\gamma},\\
    \calZ_2(\calD,\bdu)\coloneqq &\Set{(i,j,k)\in \calD : 1\le j\le j_0 \text{ and }k_0<k\le \gamma},\\ 
    \calZ_3(\calD,\bdu)\coloneqq &\Set{(i,j,k)\in \calD : 1\le j\le j_0 \text{ and }1\le k\le k_0},\\ 
    \calZ_4(\calD,\bdu)\coloneqq &\Set{(i,j,k)\in \calD : j_0< j\le \beta \text{ and }k_0<k\le \gamma},\\ 
    \calZ_5(\calD,\bdu)\coloneqq &\Set{(i,j,k)\in \calD : j_0< j\le \beta \text{ and }1\le k\le k_0},\\ 
    \calZ_6(\calD,\bdu)\coloneqq &\Set{(i,j,k)\in \calD : j>\beta \text{ and }1\le k< k_0}.  
\end{align*}
It is clear that $\calD$ is the disjoint union of the above six zones. In the subsequent discussion, they will be called $\calZ$-zones with respect to $\calD$ and $\bdu$. We will omit some of the parameters, if they are clear from the context. \Cref{Zone} gives the idea of the division of $\calD$ with respect to these zones.

\begin{figure}[h] 
\begin{center}
    \scalebox{1.4}{
\begin{tikzpicture}[thick, scale=0.5, every node/.style={scale=0.6}]
%
%
%
%
%
%
%

\draw [line width=1.2pt] (0,0)--(7,0) node [right]{$y$};
\draw [line width=1.2pt] (0,0)--(0,6) node [above]{$z$};

\draw [fill=gray!50, very thin] (2.5,2.5)--(2,2.5)--(2,2)--(2.5,2)--cycle; 
\node [above right] at (4,5) {$(j_0,k_0)$}; 
\draw [->, thin] (2.25,2.25) to [bend left] (4,5);

\path (0,0)--(0,4)  node [below left]{$\gamma$};
\path (0,-0.2)--(5,-0.2) node [below left]{$\beta$};

\draw [line width=1.2pt] (0,4)--(2,4)--(2,5)--(1.5,5)--(1.5,5.5)--(0.5,5.5)--(0.5,6)--(0,6)--cycle;
\node [right] at (0.5,4.7) {$\mathcal{Z}_1$};

\draw [line width=1.2pt] (0,2.5)--(2.5,2.5)--(2.5,4)--(0,4)--cycle;
\node [right] at (0.5,3.3) {$\mathcal{Z}_2$};

\draw [line width=1.2pt] (0,0)--(2.5,0)--(2.5,2.5)--(0,2.5)--cycle;
\node [right] at (0.5,1.2) {$\mathcal{Z}_3$};

\draw [line width=1.2pt] (2.5,2.5)--(5,2.5)--(5,3.5)--(4,3.5)--(4,4)--(2.5,4)--cycle;
\node [right] at (3.2,3.3) {$\mathcal{Z}_4$};

\draw [line width=1.2pt] (2.5,0)--(5,0)--(5,2.5)--(2.5,2.5)--cycle;
\node [right] at (3.2,1.2) {$\mathcal{Z}_5$};

\draw [line width=1.2pt]  (5,0)--(7,0)--(7,1)--(6.5,1)--(6.5,2)--(5,2)--cycle;
\node [right] at (5.2,1.2) {$\mathcal{Z}_6$};
\end{tikzpicture}}
\end{center}
\caption{$\calZ$-zones with respect to $\calD$ and $\bdu$}\label{Zone}
\end{figure}

Suppose that $a_{\calD}\ge 2$. We adopt the following induction process for considering both the vertex-decomposable property of $\Delta(\calD)$ and the primeness of $\bfI_2(\calD)$.  Consider the symmetry operation $\calS: \ZZ_+^3\to \ZZ_+^3$ by sending $(i,j,k)$ to $(i,k,j)$.  The common induction process is as follows:

\begin{enumerate}[leftmargin=3cm]
    \item [\textbf{First stage}] We remove lexicographically the initial
        points within 
        \[
            \calC\coloneqq \Set{(1,j,k)\in \calD: k\le c_{\calD^{\ge 2}}}.
        \]

    \item [\textbf{Second stage}] After we remove above points in the first stage, we do a flip by $\calS$.  The current flipped diagram also comes from $\calS(\calD)$ by removing lexicographically initial points in the $x=1$ layer:
        \[
            \calS(\calD\setminus\calC)=(\calS(\calD))_{\bdu}
        \]
       where $\bdu=(1,c_{\calD^{\ge 2}}+1,1)$. We remove lexicographically the initial points in the $x=1$ layer of the remaining flipped diagram in this stage.
\end{enumerate}

The above induction process leads to a total order on the points in the $x=1$ layer of $\calD$, and we will refer to it as the \emph{induction order}. Figure \ref{Fig:InductionOrder} gives an idea how this proceeds. The first point is $\circ=(1,1,1)$. When $c_{\calD^{\ge 2}}<c_{\calD}$, the last point is $\bullet=(1,b_{\calD}( (1,1,c_{\calD})),c_{\calD})$ of $\calD$. Otherwise, $c_{\calD^{\ge 2}}=c_{\calD}$ with the second stage disappears and the last point is $(1,b_{\calD},c_{\calD}( (1,b_{\calD},1)))$ of $\calD$.  Meanwhile, in Figure \ref{Fig:InductionOrder}, $\filledstar$ denotes the last point in the first stage while $\smallstar$ denotes the first point in the second stage.  We may also extend the induction order to the whole $\calD$ by ordering the points in $\calD^{\ge 2}$ in a suitable way and put them after the points in $\calD^1$. But this does not matter since we overall prove by induction on $a_{\calD}$.

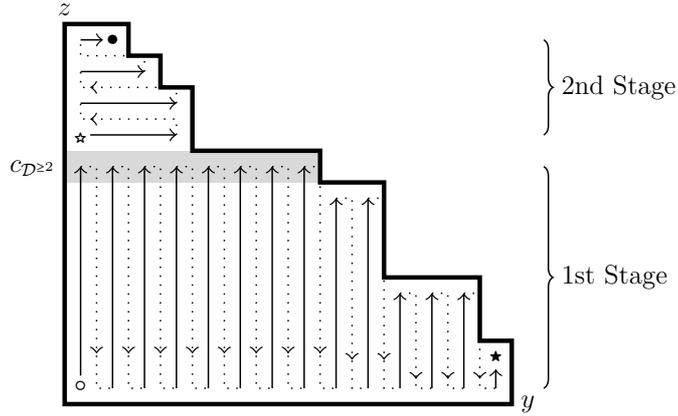
\begin{figure}[h] 
\begin{center}
    \scalebox{1.4}{
\begin{tikzpicture}[scale=0.6, every node/.style={scale=0.6}]

\fill [gray!30] (0,3.5)--(4,3.5)--(4,4)--(0,4)--cycle;

\draw [line width=1.2pt]  (0,0)--(7,0)--(7,1)--(6.5,1)--(6.5,2)--(5,2)--(5,2.5)--(5,3.5)--(4,3.5)--(4,4)--(2.5,4)--(2,4)--(2,5)--(1.5,5)--(1.5,5.5)--(1,5.5)--(1,6)--(0,6)--cycle;

\coordinate[label=right:{$y$}] (y) at (7,0);
\coordinate[label=above:{$z$}] (z) at (0,6);
\coordinate[label=below left:{$c_{\mathcal{D}^{\ge 2}}$}] (gamma) at (0,4);
\coordinate[label={$\circ$}] (O1) at (0.25,0.05);
\coordinate[label={$\filledstar$}] (T2) at (6.75,0.5);
\coordinate[label={$\bullet$}] (T1) at (0.75,5.5);
\coordinate[label={$\smallstar$}] (O2) at (0.25,3.95);

\draw [->] (0.25,0.45) -- (0.25,3.75);
\draw[ thin,  dotted,->-=.8] (0.25,3.75)--(0.5,3.75)--(0.5,0.25)--(0.75,0.25);
\draw [->] (0.75,0.25) -- (0.75,3.75);
\draw[ thin,  dotted,->-=.8] (0.75,3.75)--(1,3.75)--(1,0.25)--(1.25,0.25);
\draw [->] (1.25,0.25) -- (1.25,3.75);
\draw[ thin,  dotted,->-=.8] (1.25,3.75)--(1.5,3.75)--(1.5,0.25)--(1.75,0.25);
\draw [->] (1.75,0.25) -- (1.75,3.75);
\draw[ thin,  dotted,->-=.8] (1.75,3.75)--(2,3.75)--(2,0.25)--(2.25,0.25);
\draw [->] (2.25,0.25) -- (2.25,3.75);
\draw[ thin,  dotted,->-=.8] (2.25,3.75)--(2.5,3.75)--(2.5,0.25)--(2.75,0.25);
\draw [->] (2.75,0.25) -- (2.75,3.75);
\draw[ thin,  dotted,->-=.8] (2.75,3.75)--(3,3.75)--(3,0.25)--(3.25,0.25);
\draw [->] (3.25,0.25) -- (3.25,3.75);
\draw[ thin,  dotted,->-=.8] (3.25,3.75)--(3.5,3.75)--(3.5,0.25)--(3.75,0.25);
\draw [->] (3.75,0.25) -- (3.75,3.75);
\draw[ thin,  dotted,->-=.8] (3.75,3.75)--(4,3.75)--(4,0.25)--(4.25,0.25);
\draw [->] (4.25,0.25) -- (4.25,3.25);
\draw[ thin,  dotted,->-=.8] (4.25,3.25)--(4.5,3.25)--(4.5,0.25)--(4.75,0.25);
\draw [->] (4.75,0.25) -- (4.75,3.25);
\draw[ thin,  dotted,->-=.8] (4.75,3.25)--(5,3.25)--(5,0.25)--(5.25,0.25);
\draw [->] (5.25,0.25) -- (5.25,1.75);
\draw[ thin,  dotted,->-=.8] (5.25,1.75)--(5.5,1.75)--(5.5,0.25)--(5.75,0.25);
\draw [->] (5.75,0.25) -- (5.75,1.75);
\draw[ thin,  dotted,->-=.8] (5.75,1.75)--(6,1.75)--(6,0.25)--(6.25,0.25);
\draw [->] (6.25,0.25) -- (6.25,1.75);
\draw[ thin,  dotted,->-=.8] (6.25,1.75)--(6.5,1.75)--(6.5,0.25)--(6.75,0.25);
\draw [->] (6.75,0.25) -- (6.75,0.55);

\draw[->] (0.4,4.25)--(1.75,4.25);
\draw[thin, dotted,->-=.8] (1.75,4.25)--(1.75,4.5)--(0.25,4.5)--(0.25,4.75);
\draw[->] (0.25,4.75)--(1.75,4.75);
\draw[thin, dotted,->-=.8] (1.75,4.75)--(1.75,5)--(0.25,5)--(0.25,5.25);
\draw[->] (0.25,5.25)--(1.25,5.25);
\draw[thin, dotted] (1.25,5.25)--(1.25,5.5)--(0.25,5.5)--(0.25,5.75);
\draw[->](0.25,5.75)--(0.6,5.75);

\draw [decorate,decoration={brace, mirror}] (7.5,0.25) - - node[right] {\ 1st Stage} (7.5,3.75);
\draw [decorate,decoration={brace, mirror}] (7.5,4.25) - - node[right] {\ 2nd Stage} (7.5,5.75);
\end{tikzpicture}
}
\end{center}
\caption{Induction Order}\label{Fig:InductionOrder}
\end{figure}

\begin{Remark}
    \label{induction-order}
    \begin{enumerate}[1]
        \item The induction order is a quasi-lexicographic order.
        \item  Suppose that $\bdu=(1,j_0,k_0)$ belongs to the first stage.
            Since now $k_0\le c_{\calD^{\ge 2}}$, we have $(2,1,k_0)\in
            \calD$. For $\beta=b_{\calD}(\bdu)$, we have
            $(1,\beta+1,k_0)\notin \calD$ by definition.  Hence if $\calD$
            satisfies the projection property, then $(2,\beta+1,1)\notin \calD$.
            This implies that the zone $\calZ_6=\calZ_6^1$. 
        \item Let $\calG$ be a subdiagram of $\calD$. Under the
            operation $\calS$, the Ferrers property and the projection property of
            $\calG$ are preserved.  We extend this operation to
            $\KK[\bfT_{\calG}]$.  Obviously,
            $\calS(\bfI_2(\calG))=\bfI_2(\calS(\calG))$.  If $\bfI_2(\calG)$ is
            lexicographically quadratic, then $\calS$ also preserves the
            initial ideal with respect to the lexicographic order:
            \[
                \calS(\ini(\bfI_2({\calG})))=\ini(\bfI_2({\calS(\calG)})).
            \]
            To see this, we notice that $T_{\bdu}T_{\bdv}$ is the
            leading term of $f\coloneqq T_{\bdu}T_{\bdv}-T_{\bdu'}T_{\bdv'}\in
            \bfI_2(\calG)$ if and only if $T_{\calS(\bdu)}T_{\calS(\bdv)}$ is
            the leading term of $\calS(f)$. We also need the fact that $\calS$
            preserves the graded Hilbert function, since $\ini(\bfI_2(\calG))$
            and $\bfI_2(\calG)$ share the same graded Hilbert function as well
            as so for $\ini(\bfI_2(\calS(\calG)))$ and $\bfI_2(\calS(\calG))$.
    \end{enumerate}
\end{Remark}

\section{Cohen--Macaulayness}

The purpose of this section is to establish the Cohen--Macaulayness of the
ideal $\bfI_2(\calD)$. As advertised in the previous section, we will show that
$\Delta(\calD)$ is pure vertex-decomposable. We give readers the road map of
the proof of \Cref{VD-InductionOrder} here, because the proof is very involved.
We apply induction with respect to the order introduced in the previous section
and this is natural because the vertex-decomposable property is by definition
an induction property. During the proof, we especially focus on the deletion and
the link of the complex with respect to a particular vertex.  We pay close
attention to the generators of the link complex and its dimension because this
is where we apply induction hypothesis. Moreover, we use a lot of restriction
complexes, in the form of $\Delta(\calD,\calG)$, during the proof so that we
can reduce to a smaller case.

\begin{Theorem}
    \label{VD-InductionOrder}
    Assume that $\calD$ is a three-dimensional Ferrers diagram which
    satisfies the projection property. Then $\Delta(\calD)$ is pure
    vertex-decomposable of dimension $a_{\calD}+b_{\calD}+c_{\calD}-3$.
\end{Theorem}

\begin{proof}
   We prove by the induction on $a_{\calD}$. When $a_{\calD}=1$, this can be reduced to the two-dimensional case in \cite[Theorem 3.3]{CNPY}. Now, we assume that $a_{\calD}\ge 2$. To finish the proof, we will proceed by removing the points in the $x=1$ layer according to the induction order given in the previous section.  Write $\calC\coloneqq \Set{(1,j,k)\in \calD:k\le c_{\calD^{\ge 2}}}$.

   \bigskip
   \textbf{First stage} 

   We study the subdiagram $\calA_{\bdu}$ for each $\bdu\in \calC$, which is the subdiagram of $\calD$ by removing points preceeding $\bdu$ with respect to the induction order.  By
   induction, \Cref{restriction-QLO} and \Cref{equivalent-definition},
   $\Delta(\calD,\calA_{\bdu})=\Delta(\calA_{\bdu})$ is expected to have
   dimension
   \begin{equation}
       a_{\calD^{\ge 2}}+b_{\calD^{\ge 2}}+c_{\calD^{\ge 2}}-3+ \left|
           \Phan(\calD)\cap \calA_{\bdu}  \right|. \label{dim-Au}
   \end{equation}
   We will show that $\Delta(\calA_{\bdu})$ is pure vertex-decomposable.

   The minimal case is when we remove all the points in $\calC$. We will deal
   it in the second stage.  Other than this minimal case, when $\bdu$ is a
   phantom point, this case is clear by \Cref{VD-cone} and
   \Cref{equivalent-definition}.  Thus, we may assume that $\bdu$ is a normal
   point.  Say that $\bdu=(1,j_0,k_0)$ with 
   \[
       \alpha=a_{\calD}(\bdu),\quad 
       \beta=b_{\calD}(\bdu)\quad \text{and} \quad
       \gamma=c_{\calD}(\bdu).  
   \]
   We only need to check that $T_{\bdu}$ is the expected shedding vertex for
   the simplicial complex. Note that from our induction hypothesis, the
   deletion complex $\Delta(\calA_{\bdu})\setminus
   T_{\bdu}=\Delta(\calA_{\bdu}^+)$ is pure vertex-decomposable of the expected
   dimension given by \eqref{dim-Au}.  As for the link complex
   $\link_{\Delta(\calA_{\bdu})}(T_{\bdu})$, its Stanley--Reisner ideal in
   $\KK[\bfT_{\calA_{\bdu}^+}]$ is the colon ideal
   $I_{\Delta(\calA_{\bdu})}:T_{\bdu}$. We need to investigate this colon ideal
   in detail.  Since the monomial generators of $I_{\Delta(\calA_{\bdu})}$ are
   all quadratic, the monomial generators of the colon ideal
   $I_{\Delta(\calA_{\bdu})}:T_{\bdu}$ are either quadratic or linear.

   As observed in \Cref{restriction-QLO}, $\bdu$ satisfies the ``leading
   monomial'' condition in \Cref{I2-inheritance}.  Therefore, $T_{\bdu'}\in
   I_{\Delta(\calA_{\bdu})}:T_{\bdu}$ if and only if $T_{\bdu}$ and $T_{\bdu'}$
   can switch coordinates.  Now, the lattice points in $\calA_{\bdu}$ that
   contribute the linear minimal generators of
   $I_{\Delta(\calA_{\bdu})}:T_{\bdu}$ come from the following ``linear''
   regions:
   \begin{align}
       \notag &\Set{(1,j,k)\in \calD: j> j_0 \text{ and } k>k_0}             & \text{ by switching the $y$-coordinates,}\\ 
       \notag &\Set{(i,j,k)\in \calD: i\ge 2  \text{ and } j_0<j\le \beta}   & \text{ by switching the $y$-coordinates,}\\ 
       \notag &\Set{(i,j,k)\in \calD: i\ge 2  \text{ and } k_0<k\le \gamma}  & \text{ by switching the $z$-coordinates,}\\ 
       \tag{$\dag$} &\Set{(i,j,k)\in \calD: j>\beta \text{ and } 2\le i\le \alpha}&\text{ by switching the $x$-coordinates.}
   \end{align}
   However, since $k_0\le c_{\calD^{\ge 2}}$, $\calZ_6=\calZ_6^1$ by the
   projection property. In particular, the last region $(\dag)$ is indeed
   empty.

   Consequently, we see that
   \begin{align*}
       I_{\Delta(\calA_{\bdu})}:T_{\bdu}= (T_{\bdu'}:
       \bdu'\in \calZ_4\cup\calZ_2^{\ge 2}\cup \calZ_5^{\ge
           2})+I_{\Delta(\calA_{\bdu})}\cap \KK[T_{\bdu'}:\bdu'\in \calE].
   \end{align*}
   Here, 
   \[
       \calE\coloneqq \calA_{\bdu}^{+}\setminus (\calZ_4\cup\calZ_2^{\ge 2}\cup
       \calZ_5^{\ge 2}).
   \] 
   So far, $\Delta(\calD,\calE)$ is the restriction complex of the link complex
   $\link_{\Delta(\calD,\calA_{\bdu})}(T_{\bdu})$ by removing the ``invisible
   vertices'' in $(T_{\bdu'}: \bdu'\in \calZ_4\cup\calZ_2^{\ge 2}\cup
   \calZ_5^{\ge 2})$.  Notice that
   $\Delta(\calD,\calA_{\bdu})=\Delta(\calA_{\bdu})$ and $\bdu$ is a normal
   point. As we hope that $T_{\bdu}$ is a shedding vertex at this step, we are
   reduced to show that $\Delta(\calD,\calE)$ is a pure vertex-decomposable
   complex of dimension
   \begin{equation}
       \dim \Delta(\calD,\calE)=\dim \Delta(\calA_{\bdu})-1.
       \label{E-Au}
   \end{equation}

   Notice that for each $(1,j,k)\in \calE$ with $1\le j\le j_0$, we have
   \begin{itemize}
       \item $j=j_0$ and $k_0< k\le \gamma$, or
       \item $j<j_0$ and $k>c_{\calD^{\ge 2}}$.  
   \end{itemize}
   Let 
   \begin{align*}
       \calH\coloneqq &\calE\setminus \Set{(1,j_0,k): k_0<k\le
           \min(\gamma,c_{\calD^{\ge 2}})}\\
       =&\calZ_1^{\ge 2} \cup \calZ_3^{\ge 2} \cup \calZ_5^1\cup \calZ_6^1
       \cup \Set{(1,j,k)\in \calD: j\le j_0 \text{ and } k>c_{\calD^{\ge 2}}}.
   \end{align*}
   For $\bar{\bdu}\in \Set{(1,j_0,k):k_0<k\le \min(\gamma,c_{\calD^{\ge 2}})}$,
   we can similarly define $\bar{\calH}$. Notice that $\calH\subseteq
   \bar{\calH}$.  This implies that for any $\bdv\in \calH$, one has
   $T_{\bar{\bdu}}T_{\bdv}\notin I_{\Delta(\calA_{\bdu})}$. Therefore,
   $\Delta(\calD,\calE)$ is the join of $\Delta(\calD,\calH)$ with a simplex of
   dimension $\min(\gamma,c_{\calD^{\ge 2}})-k_0-1$. Consequently
   \begin{equation}
       \dim(\Delta(\calD,\calH))=
       \dim(\Delta(\calD,\calE))-(\min(\gamma,c_{\calD^{\ge 2}})-k_0), 
       \label{EH}
   \end{equation}
   and by \Cref{I2-inheritance-application},
   \begin{align*}
       \gens(I_{\Delta(\calD)}\cap \KK[T_{\bdu'}:\bdu'\in \calE])=
       \gens(I_{\Delta(\calA_{\bdu})}\cap \KK[T_{\bdu'}:\bdu'\in \calE])\\
       =\gens(I_{\Delta(\calD,\calE)})=\gens(I_{\Delta(\calD,\calH)}).
   \end{align*}
   To summarize, by combining \eqref{dim-Au}, \eqref{E-Au} and \eqref{EH}, we
   are expecting
   \begin{equation}
       \dim(\Delta(\calD,\calH)) = a_{\calD^{\ge 2}}+b_{\calD^{\ge
               2}}+c_{\calD^{\ge 2}}-3+ \left| \Phan(\calD)\cap \calA_{\bdu}
       \right| -(\min(\gamma,c_{\calD^{\ge 2}})-k_0+1).
       \label{dim-H-InductionOrder}
   \end{equation}
   Now, it suffices to show that $\Delta(\calD,\calH)$ is pure
   vertex-decomposable of this expected dimension. We prove this in
   \Cref{calculation-lemma}.

   \bigskip
   \textbf{Second stage}

   After the first stage, we are dealing with the case where we removed all the points in $\calC$.  Now we flip $\calD$ to get $\calS(\calD)$, which will be written as $\calD'$ for simplicity.  Notice that for the remaining points in $(\calD')^1$, the induced induction order is exactly the lexicographic order. Now, as before, we will write $\calD'_{\bdu}$ for the restriction diagram from $\calD'$ by removing those points preceding $\bdu$ lexicographically in $\calD'$. Using this notation, the current case is $\calS(\calD\setminus\calC)=\calD'_{(1,c_{\calD^{\ge 2}}+1,1)}$, since by flipping the diagram, we switch the $y$ and $z$ coordinates.  By \Cref{induction-order},
   \[
       \Delta(\calS(\calD),\calS(\calD\setminus\calC))\cong \Delta(\calD'_{(1,c_{\calD^{\ge 2}}+1,1)}).
   \]
   In the following, we will remove the points in the $x=1$ layer of $\calD'_{(1,c_{\calD^{\ge 2}}+1,1)}$ lexicographically, and prove the corresponding complex is pure vertex-decomposable of expected dimension.

   The minimal case is when we remove all the $x=1$ layer. Now, we have $(\calD')^{\ge 2}=\calS(\calD^{\ge 2})$. By induction, $\Delta( (\calD')^{\ge 2})$ is pure vertex-decomposable of dimension 
   \[
       a_{\calD^{\ge 2}}+b_{\calD^{\ge 2}}+c_{\calD^{\ge 2}}-3.
   \]

   Now, consider a general $\bdu=(1,j_0,k_0)\in \calD'_{(1,c_{\calD^{\ge 2}}+1,1)}$.  By induction,  \Cref{restriction-QLO} and \Cref{equivalent-definition}, $\Delta(\calD'_{\bdu})$ has dimension
   \begin{equation}
       a_{\calD'^{\ge 2}}+b_{\calD'^{\ge 2}}+c_{\calD'^{\ge 2}}-3+ \left|
           \Phan(\calD')\cap \calD'_{\bdu}  \right|. 
       \label{dim-Du}
   \end{equation}
   Notice that $\Phan(\calD')\cap \calD'_{\bdu}$ coincides with $\Phan(\calD'_{\bdu})$ by \Cref{I2-inheritance-application}.  We will show that $\Delta(\calD'_{\bdu})$ is pure vertex-decomposable.

   For general $\calD'_{\bdu}$, when $\bdu$ is a phantom point, this case is clear by \Cref{equivalent-definition}.  Thus, we may assume that $\bdu$ is not a phantom point.  Say that  
   \[
       \alpha\coloneqq a_{\calD'}(\bdu),\quad 
       \beta\coloneqq b_{\calD'}(\bdu)\quad \text{and} \quad
       \gamma\coloneqq c_{\calD'}(\bdu).  
   \]
   We will check that $T_{\bdu}$ is the expected shedding vertex for the simplicial complex. For this, we prove similarly as in the first stage.

   Notice that, in the first stage, the ``ceiling restriction'' of choosing points $(1,j,k)$ with $k\le c_{\calD^{\ge 2}}$ is mainly used to ensure that $\calZ_6(\calD)=\calZ_6^1(\calD)$.  In the current case we will automatically get $\calZ_6(\calD')=\calZ_6^1(\calD')$. In the following, the $\calZ$-zones are with respect to $\calD'$.

   Now, as in the first stage, it suffices to show that the link complex $\link_{\Delta(\calD'_{\bdu})}(T_{\bdu})$ is pure vertex-decomposable of dimension one less. By a similar screening, we see the colon ideal
   \begin{align*}
       I_{\Delta(\calD'_{\bdu})}:T_{\bdu}= (T_{\bdu'}: \bdu'\in
       \calZ_4\cup\calZ_2^{\ge 2}\cup \calZ_5^{\ge
           2})+I_{\Delta(\calD'_{\bdu})}\cap \KK[T_{\bdu'}:\bdu'\in \calE'].
   \end{align*}
   Here, for $(\calD_{\bdu}')^+:=\calD'_{\bdu}\setminus \bdu$, we have 
   \[
       \calE'\coloneqq (\calD_{\bdu}')^+\setminus (\calZ_4\cup\calZ_2^{\ge 2}\cup
   \calZ_5^{\ge 2}).
   \]
   Notice that for each $(1,j,k)\in \calE'$ with $1\le j\le j_0$, we have $j=j_0$ and $k_0< k\le \gamma$.  Let $\calH'\coloneqq \calE'\setminus \calZ_2^1$.  Similarly, we see that $\Delta(\calD',\calE')$ is the join of $\Delta(\calD',\calH')$ with a simplex of dimension 
   \begin{equation*}
       \gamma-k_0-1,
   \end{equation*}
   and by \Cref{I2-inheritance-application},
   \begin{align*}
       &\gens(I_{\Delta(\calD')}\cap \KK[T_{\bdu'}:\bdu'\in \calE'])=
       \gens(I_{\Delta(\calD'_{\bdu})}\cap \KK[T_{\bdu'}:\bdu'\in \calE'])\\
       =&\gens(I_{\Delta(\calD',\calE')})=\gens(I_{\Delta(\calD',\calH')}).
   \end{align*}
   By the previous discussion, we are similarly anticipating
   \begin{align}
       \notag \dim(\Delta(\calD',\calH'))=& \dim(\Delta(\calD',\calE'))-(\gamma-k_0) =\dim(\Delta (\calD'_{\bdu}))-(\gamma-k_0+1)\\
       \label{dim-H} =& a_{\calD'^{\ge 2}}+b_{\calD'^{\ge 2}}+c_{\calD'^{\ge 2}}-3+ \left| \Phan(\calD')\cap \calD'_{\bdu}  \right| -(\gamma-k_0+1).
   \end{align}
   Now, it suffices to show that $\Delta(\calD',\calH')$ is pure
   vertex-decomposable of this expected dimension. We prove this in
   \Cref{calculation-lemma-Dprime}.
\end{proof}

\begin{Lemma}
    \label{restriction-DH}
    Using the notation in the proof of \Cref{VD-InductionOrder}, we have  
    \[
        \bfI_2(\calD)\cap \KK[\bfT_{\calH}]=\bfI_2(\calH) \quad \text{ and }
        \quad  \gens(\ini(\bfI_2(\calD)))\cap
        \KK[\bfT_{\calH}]=\gens(\ini(\bfI_2(\calH))).  
    \]
    In particular, $\Delta(\calD,\calH)=\Delta(\calH)$ and $\bfI_2(\calH)$ is
    lexicographically quadratic.
\end{Lemma}

\begin{proof}
    Notice that $\calH\subseteq \calA_{\bdu}^{+}$.  Since we already have
    similar formulas for $\calA_{\bdu}^{+}$ instead of $\calH$ in
    \Cref{restriction-QLO}, we may first replace $\calD$ by
    $\calA_{\bdu}^{+}$.  Now, it suffices to verify directly that when
    $\calH\ne \calA_{\bdu}^{+}$, the pair $\calH\subset \calA_{\bdu}^{+}$
    satisfies the detaching condition in \Cref{I2-inheritance}.
\end{proof}

\begin{Lemma}
    \label{calculation-lemma}
        Using the notation in the proof of \Cref{VD-InductionOrder}, the complex $\Delta(\calD,\calH)=\Delta(\calH)$ is pure
    vertex-decomposable of dimension given by \Cref{dim-H-InductionOrder}.
\end{Lemma}

\begin{proof}
    Let 
    \[
        \widetilde{\calD}\coloneqq \calZ_3 \cup \calZ_5^1 \cup \calZ_6
        \cup \Set{(i,j,k)\in \calD: j\le j_0 \text{ and }
            k>\min(\gamma,c_{\calD^{\ge 2}})}.
    \]
    This diagram is essentially a three-dimensional Ferrers diagram
    which satisfies the projection property. Furthermore, notice that
    $\calZ_6=\calZ_6^1$ and $\calZ_1\subset \widetilde{\calD}$.  Like
    $\calA_{\bdu}$ and $\calA_{\bdu}^{+}$ for $\calD$, we similarly define
    $\widetilde{\calA}_{\bdu}$ and $\widetilde{\calA}_{\bdu}^{+}$ for
    $\widetilde{\calD}$.  As the first step, we notice that
    $\calH=\widetilde{\calA}_{\bdu}^{+}$. Even when
    $\calD=\widetilde{\calD}$, this is a strictly smaller case compared to
    $\calA_{\bdu}$. Hence by induction, $\Delta(\calH)$ is pure
    vertex-decomposable of dimension
    \begin{equation}
        a_{\widetilde{\calD}^{\ge 2}}+b_{\widetilde{\calD}^{\ge
                2}}+c_{\widetilde{\calD}^{\ge 2}}-3+ \left|
            \Phan(\widetilde{\calD})\cap \calH \right|. \label{dim-H1}
    \end{equation}
    It remains to verify that this number agrees with
    \eqref{dim-H-InductionOrder}.

    Notice that by \eqref{x1-phantonpoints-number}, we have
    \begin{align*}
        \left| \Phan(\calD)\cap \calA_{\bdu} \right|=&
        \left| \Phan(\calD) \right| - \left| \Phan(\calD)\setminus
            \calA_{\bdu} \right| \\
        =& (b_{\calD}+c_{\calD})- (b_{\calD^{\ge
                2}}+c_{\calD^{\ge 2}})+1 - \left| \Phan(\calD)\setminus
            \calA_{\bdu} \right|,
    \end{align*}
    and similarly
    \[
        \left| \Phan(\widetilde{\calD})\cap \calH \right|=
        (b_{\widetilde{\calD}}+c_{\widetilde{\calD}})- (b_{\widetilde{\calD}^{\ge
                2}}+c_{\widetilde{\calD}^{\ge 2}})+1 - \left|
            \Phan(\widetilde{\calD})\setminus \calH \right|.
    \]
    Since
    \[
        a_{\calD^{\ge 2}}=a_{\widetilde{\calD}^{\ge 2}},\quad 
        b_{\calD}=b_{\widetilde{\calD}}\quad \text{and} \quad
        c_{\calD}-(\min(\gamma,c_{\calD^{\ge
                2}})-k_0)=c_{\widetilde{\calD}},
    \]
    we are reduced to show that
    \begin{equation}
        \left| \Phan(\widetilde{\calD})\setminus \calH \right|=
        \left| \Phan(\calD)\setminus \calA_{\bdu} \right|+1.
        \label{last-check}
    \end{equation}
    Notice that if $\bdv=(1,j,k)\in \Phan(\calD)\setminus\calA_{\bdu}$ or
    $\Phan(\widetilde{\calD})\setminus\calH$, then $j\le j_0$.  
    \begin{enumerate}[a]
        \item Consider the case when $j<j_0$.
            \begin{enumerate}[i]
                \item Suppose that $c_{\calD^{\ge 2}}<\gamma$. It is clear that
                    \begin{itemize}
                        \item $\bdv$ is in the border of $\calD$ if and only
                            if it is in the border of $\widetilde{\calD}$;
                        \item $\bdv$ is a phantom point with respect to
                            $\calD$ if and only if it is so with respect to
                            $\widetilde{\calD}$ by \Cref{phantom-x1}, and
                        \item when $\bdv$ is a common phantom point, then
                            $\bdv$ belongs to $\calA_{\bdu}$ if and only if it
                            belongs to $\calH$.
                    \end{itemize}
                \item Suppose that $c_{\calD^{\ge 2}}\ge \gamma$. We may
                    simply assume that $k_0=\gamma$ and argue as above.
            \end{enumerate}

        \item Consider the case for $j=j_0$. If $\bdv$ is a border point of
            $\calD^1$ not belonging to $\calA_{\bdu}$, then $k<k_0\le
            c_{\calD^{\ge 2}}$.  Hence $\bdv$ is a normal point, since a
            $z$-coordinates switch is always feasible. This means that we have
            no such phantom point in $\Phan(\calD)\setminus \calA_{\bdu}$. 

            On the other hand, with respect to $\widetilde{\calD}$, we first
            notice that
            \[
                c_{\widetilde{\calD}}( (1,j_0+1,1))=
                \min(c_{\calD}( (1,j_0+1,1)),k_0). 
            \]
            Each border point $\bdv=(1,j_0,k)\in \widetilde{\calD}$ satisfies
            $c_{\widetilde{\calD}}( (1,j_0+1,1))\le k \le k_0$ or
            $\min(\gamma,c_{\calD^{\ge 2}})< k\le \gamma$. 
            \begin{itemize}
                \item When $c_{\widetilde{\calD}}( (1,j_0+1,1))\le k < k_0$,
                    a $z$-coordinates switch is available. Hence the border point
                    is a normal point.  
                \item When $\min(\gamma,c_{\calD^{\ge 2}})< k\le \gamma$, the
                    border point belongs to $\calH$.  
            \end{itemize}
            Hence, it suffices to show that the border point $\bdu=(1,j_0,k_0)$ is a
            phantom point for $\widetilde{\calD}$.  Notice that for any
            $(1,j_1,k_1)\in \widetilde{\calA}_{\bdu}^{+}$, one has
            $(2,j_1,k_1)\notin \widetilde{\calD}$.  Thus, $x$-coordinates
            switch is forbidden for $\bdu$ in $\widetilde{\calA}_{\bdu}$.
            Since $(2,j_0+1,1)\notin \widetilde{\calD}$, $y$-coordinates switch
            is also not possible for $\bdu$ in $\widetilde{\calA}_{\bdu}$.
            \begin{itemize} 
                \item When $c_{\calD^{\ge 2}}\ge \gamma$, we cannot increase
                    the $z$-coordinate of $\bdu$ within $\widetilde{\calD}$.  
                \item  When $c_{\calD^{\ge 2}}<\gamma$, we have
                    $c_{\widetilde{\calD}^{\ge 2}}=k_0$.
            \end{itemize}
            In either case, $z$-coordinates switch is not possible for
            $\bdu$ in $\widetilde{\calA}_{\bdu}$.  Thus, $\bdu$ is a phantom
            point for $\widetilde{\calD}$.  This only phantom point
            contributes to the number $1$ in \eqref{last-check}. 
    \end{enumerate}

    Thus, we have established \eqref{last-check}.
\end{proof}

\begin{Lemma}
	    Using the notation in the proof of \Cref{VD-InductionOrder}, we have  
    \[
        \bfI_2(\calD')\cap \KK[\bfT_{\calH'}]=\bfI_2(\calH') \quad \text{ and
        } \quad  \gens(\ini(\bfI_2(\calD')))\cap
        \KK[\bfT_{\calH'}]=\gens(\ini(\bfI_2(\calH'))).  
    \]
    In particular, $\Delta(\calD',\calH')=\Delta(\calH')$ and $\bfI_2(\calH')$
    is lexicographically quadratic.
\end{Lemma}

\begin{proof}
    The proof is similar to that of \Cref{restriction-DH}.  
\end{proof}

\begin{Lemma}
    \label{calculation-lemma-Dprime}
        Using the notation in the proof of \Cref{VD-InductionOrder},  
    the complex $\Delta(\calD',\calH')=\Delta(\calH')$ is pure
    vertex-decomposable of dimension given by \Cref{dim-H}.
\end{Lemma}

\begin{proof}
    The $\calZ$-zones here are with respect to $\calD'=\calS(\calD)$.  Let 
    \[
        \widetilde{\calD'}\coloneqq \calZ_1\cup\calZ_3 \cup \calZ_5 \cup \calZ_6.
    \]
    This diagram is still essentially a three-dimensional Ferrers diagram which satisfies the projection property. Furthermore, notice that $\calZ_5=\calZ_5^1$ and $\calZ_6=\calZ_6^1$ since $b_{(\calD')^{\ge 2}}=c_{\calD^{\ge 2}}$.  As in the first step, we notice that $\calH'=\widetilde{\calD'}_{(1,j_0+1,1)}$, the subdiagram of $\widetilde{\calD'}$ by removing the points preceding $(1,j_0+1,1)$ lexicographically.  Even when $\calD'=\widetilde{\calD'}$, this is a strictly smaller case compared with $\calS(\calD\setminus\calC)$.  Hence by induction, $\Delta(\calH')$ is pure vertex-decomposable of dimension
    \begin{equation}
        a_{\widetilde{\calD'}^{\ge 2}}+b_{\widetilde{\calD'}^{\ge
                2}}+c_{\widetilde{\calD'}^{\ge 2}}-3+ \left|
            \Phan(\widetilde{\calD'})\cap \calH' \right|.
        \label{dim-H1prime}
    \end{equation}
    It remains to verify that this number agrees with
    \eqref{dim-H}.

    Notice that by \eqref{x1-phantonpoints-number}, we have
    \begin{align*}
        \left| \Phan(\calD')\cap \calD'_{\bdu} \right|=&
        \left| \Phan(\calD') \right| - \left| \Phan(\calD')\setminus
           \calD'_{\bdu} \right| \\
        =& (b_{\calD'}+c_{\calD'})- (b_{\calD'^{\ge
                2}}+c_{\calD'^{\ge 2}})+1 - \left| \Phan(\calD')\setminus
           \calD'_{\bdu} \right|,
    \end{align*}
    and similarly
    \[
        \left| \Phan(\widetilde{\calD'})\cap \calH' \right|=
        (b_{\widetilde{\calD'}}+c_{\widetilde{\calD'}})- (b_{\widetilde{\calD'}^{\ge
                2}}+c_{\widetilde{\calD'}^{\ge 2}})+1 - \left|
            \Phan(\widetilde{\calD'})\setminus \calH' \right|.
    \]
    Since
    \[
        a_{\calD'^{\ge 2}}=a_{\widetilde{\calD'}^{\ge 2}},\quad 
        b_{\calD'}=b_{\widetilde{\calD'}}\quad \text{and} \quad
        c_{\calD'}-(\gamma-k_0)=c_{\widetilde{\calD'}},
    \]
    we are reduced to show that
    \begin{equation}
        \left| \Phan(\widetilde{\calD'})\setminus \calH' \right|=
        \left| \Phan(\calD')\setminus \calD'_{\bdu} \right|+1.
        \label{last-check-Dprime}
    \end{equation}
    The proof is similar to but simpler than that for
    \Cref{calculation-lemma}.  If $\bdv=(1,j,k)\in
    \Phan(\calD')\setminus\calD'_{\bdu}$ or
    $\Phan(\widetilde{\calD'})\setminus\calH'$, then $j=j_0$. 
   
    Take arbitrary $\bdv\in \Phan(\calD')\setminus \calD'_{\bdu}$. Then,
    \begin{itemize}
        \item $k< k_0$ since $\bdv\notin \calD'_{\bdu}$;
        \item $k\ge \gamma'\coloneqq c_{\calD'}( (1,j_0+1,1))$ since $\bdv$ is a border point;
        \item $k\ge c_{(\calD')^{\ge 2}}$ by \Cref{phantom-x1}.
    \end{itemize}
    It is also clear that any $(1,j_0,k)$ satisfying the above requirements
    belongs to $\Phan(\calD')\setminus \calD'_{\bdu}$.  The cardinality of
    this set is
    \begin{equation}
        \max\left( k_0-\max(\gamma',c_{(\calD')^{\ge 2}}),0 \right).  
        \label{phan-Du}
    \end{equation}

    Likewise, $\bdv=(1,j_0,k) \in \Phan(\widetilde{\calD'})\setminus\calH'$ if and only if the following three requirements are satisfied.
    \begin{enumerate}[a]
        \item \label{a} $k_0\ge k$ since $c_{\widetilde{\calD'}}((1,j_0,1))=k_0$.
        \item $k\ge \min(\gamma',k_0)$ since $c_{\widetilde{\calD'}}((1,j_0+1,1))=\min(\gamma',k_0)$.
        \item When $c_{{\calD'}^{\ge 2}}<k_0$, $c_{\widetilde{\calD'}^{\ge 2}}= c_{{\calD'}^{\ge 2}}$ and the only additional requirement for $\bdv$ is
            \[
                k\ge c_{(\calD')^{\ge 2}}
            \]
            by \Cref{phantom-x1}.  On the other hand, when $c_{{\calD'}^{\ge
                    2}}<k_0$, the corresponding requirement is
            \[
                k=k_0.
            \] 
            Taking account of \ref{a} above, we can combine these two into
            the condition:
            \[
                k\ge \min(c_{(\calD')^{\ge 2}},k_0).
            \]
    \end{enumerate}
    Thus the cardinality of $\Phan(\widetilde{\calD'})\setminus\calH'$ is
    \begin{equation}
        k_0+1-\max\left( \min(\gamma',k_0),\min(c_{(\calD')^{\ge 2}},k_0) \right).
        \label{phan-H}
    \end{equation}
    One can verify directly that \eqref{phan-H} minus \eqref{phan-Du} equals
    one, which corresponds to the $1$ in \eqref{last-check-Dprime}. Thus, we
    have established \eqref{last-check-Dprime}.
\end{proof}

It follows from the proof of \Cref{VD-InductionOrder} and the discussion in
\Cref{section:Cohen--Macaulay} that

\begin{Corollary}
    Let $\calD$ be a three-dimensional Ferrers diagram which satisfies the
    projection property.  Take arbitrary $\bdu\in \calD^1$. Then both
    $\bfI_2(\calA_{\bdu})$ and $\ini(\bfI_2(\calA_{\bdu}))$ are
    Cohen--Macaulay of the same codimension given by \eqref{dim-Au}. In
    particular,  both $\bfI_2(\calD)$ and $\ini(\bfI_2(\calD))$ are
    Cohen--Macaulay.
\end{Corollary}

\section{Primeness} 

Let $\calD$ be a three-dimensional Ferrers diagram which satisfies the
projection property. The theme of this section is to show $\bfI_2(\calD)$ is
a prime ideal. The primary strategy is to use the Cohen--Macaulayness shown
in the previous section. In particular, the ideal $\bfI_2(\calD)$ is unmixed.
We also need to define suitable maps in view of the localization of
variables, so that we can proceed by induction. 

\begin{Lemma}
    \label{prime-induction}
    Let $\calD$ be a finite three-dimensional diagram.  Take arbitrary
    $\bdu\in \calD$ and let $\calD'=\calD\setminus \bdu$.  Suppose that
    $\bfI_2(\calD)$ is unmixed and
    $\codim(\bfI_2(\calD))=\codim(\bfI_2(\calD'))+1$.  If the localization
    $(\KK[\bfT_{\calD}]/\bfI_2(\calD))[T_{\bdu}^{-1}]$ is a domain and
    $\bfI_2(\calD')$ is a prime ideal, then $\bfI_2(\calD)$ is also a prime
    ideal.
\end{Lemma}

\begin{proof}
    Consider the ideal $\bfI_2(\calD)$ in $S=\KK[\bfT_{\calD}]$, whose
    associated primes are $\frakp_1,\dots,\frakp_m$. Since $T_{\bdu}$
    does not vanish at the point $(1,1,\dots,1)\in \KK^{|\calD|}$
    which is a zero of $\bfI_2(\calD)$, $T_u\notin
    \sqrt{\bfI_2(\calD)}=\frakp_1\cap\cdots\cap\frakp_m$. Meanwhile,
    $\bfI_2(\calD)S[T_u^{-1}]$ is a prime ideal by hypothesis. Hence,
    without loss of generality, we may assume that $T_u\notin\frakp_1$
    while $T_{u}\in \frakp_i$ for $i\ne 1$.

    When $m=1$, $\bfI_2(\calD)$ is $\frakp_1$-primary. Hence 
    \[
        \bfI_2(\calD)=\bfI_2(\calD) S_{\frakp_1}\cap S\supseteq
        \bfI_2(\calD)S[T_{u}^{-1}]\cap S\supseteq \bfI_2(\calD),
    \]
    which means $\bfI_2(\calD)$ is a prime ideal.

    When $m\ge 2$, by the unmixedness assumption,
    \begin{equation} 
        \codim(\bfI_{2}(\calD))=\codim(\bfI_2(\calD),T_{\bdu}).
        \label{SameHeight}
    \end{equation} 
    Since $\bfI_2(\calD)\ne \bfI_2(\calD')$ by our hypothesis, we
    can find some quadratic binomial $f=T_{\bdu}T_{\bdv}-T_{\bdu'}T_{\bdv'}
    \in \bfI_2(\calD)\setminus \bfI_2(\calD')$. Obviously, $\bdu',\bdv'\in
    \calD'$ and the ideal $(f,T_{\bdu})=(T_{\bdu'}T_{\bdv'},T_{\bdu})$.  
    Since $\bfI_2(\calD')$ is a prime ideal and $T_{\bdu'}T_{\bdv'}\notin J_{\calD'}\supseteq \bfI_2(\calD')$,
    \[
        \codim(\bfI_2(\calD'),T_{\bdu'}T_{\bdv'})=\codim(\bfI_2(\calD'))+1=\codim(\bfI_2(\calD)).
    \]
    However,
    \begin{align*}
        &\codim(\bfI_2(\calD),T_{\bdu})\ge
        \codim(\bfI_2(\calD'),f,T_{\bdu}) =
        \codim(\bfI_2(\calD'),T_{\bdu'}T_{\bdv'},T_{\bdu}) \\
        =&\codim(\bfI_2(\calD'),T_{\bdu'}T_{\bdv'})+1
        =\codim(\bfI_2(\calD))+1.
    \end{align*}
    This is a contradiction to \eqref{SameHeight}. 
    
    Therefore, $m=1$ and the ideal $\bfI_2(\calD)$ is a prime ideal.
\end{proof}

\begin{Theorem}
    \label{I2prime}
    Let $\calD$ be a three-dimensional Ferrers diagram which satisfies the
    projection property. Then $\bfI_2(\calD)$ is a prime ideal in
    $\KK[\bfT_{\calD}]$.
\end{Theorem}

\begin{proof}
    We prove by the induction on $a_{\calD}$. When $a_{\calD}=1$, this can be
    reduced to the two-dimensional case in \cite[Proposition 3.5]{CNPY}. Thus,
    in the following, we assume that $a_{\calD}\ge 2$. We proceed by removing
    the points in the $x=1$ layer, using the induction order given in Section
    3. Recall that $\calC\coloneqq \Set{(1,j,k)\in \calD:k\le c_{\calD^{\ge
                2}}}$.

    \bigskip
    \textbf{First stage} 

    Suppose that $\bdu=(i_0=1,j_0,k_0)\in \calC$. Recall that $\calA_{\bdu}$ is
    obtained from $\calD$ by removing those points in $\calC$ that are
    lexicographically before $\bdu$. And
    $\calA_{\bdu}^{+}=\calA_{\bdu}\setminus \bdu$.  We want to prove that
    $\bfI_2(\calA_{\bdu})$ is a prime ideal.  The minimal case is when we
    remove all the points in $\calC$. We will deal it in the second stage.
    Other than this minimal case, by induction, we may assume that
    $\bfI_2(\calA_{\bdu}^{+})$ is a prime ideal.  We may assume that
    $\bfI_2(\calA_{\bdu}^{+})\ne \bfI_2(\calA_{\bdu})$. Whence, $\bdu$ is a
    normal point and $\codim \bfI_2(\calA_u)=\codim \bfI_2(\calA_u^{+})+1$.
    Using \Cref{HeightFormula}, \Cref{VD-InductionOrder} and its proof, we are
    reduced to showing that
    $(\KK[\bfT_{\calA_{\bdu}}]/\bfI_2(\calA_{\bdu}))[T_{\bdu}^{-1}]$ is a
    domain, by \Cref{prime-induction}.

    Consider the $\KK$-algebra homomorphism 
    $
    \varphi:(\KK[\bfT_{\calA_{\bdu}}])[T_{\bdu}^{-1}]\to
    (\KK[\bfT_{\calA_{\bdu}}])[T_{\bdu}^{-1}] 
    $
    defined by 
    \[
        T_{ijk}\mapsto 
        \begin{cases}
            T_{ijk}+T_{ijk_0}T_{i_0j_0k}T_{\bdu}^{-1}, & \text{ if $(i,j,k)\in \calZ_2^{\ge 2}(\calA_{\bdu})$,}\\
            T_{ijk}+T_{ij_0k_0}T_{i_0jk_0}T_{i_0j_0k}T_{\bdu}^{-2}, & \text{ if $(i,j,k)\in \calZ_4(\calA_{\bdu})$,}\\
            T_{ijk}+T_{ij_0k}T_{i_0jk_0}T_{\bdu}^{-1}, & \text{ if $(i,j,k)\in \calZ_5^{\ge 2}(\calA_{\bdu})$,}\\
            T_{ijk}, & \text{ otherwise.}
        \end{cases}
    \]
    Here, by $\calZ_*(\calA_{\bdu})$, we mean $\calZ_*(\calD)\cap \calA_{\bdu}$.
    And the six zones are partitioned with respect to $\bdu$.  The above map
    gives an isomorphism whose inverse map is
    \[
        T_{ijk}\mapsto 
        \begin{cases}
            T_{ijk}-T_{ijk_0}T_{i_0j_0k}T_{\bdu}^{-1}, & \text{ if $(i,j,k)\in \calZ_2^{\ge 2}(\calA_{\bdu})$,}\\
            T_{ijk}-T_{ij_0k_0}T_{i_0jk_0}T_{i_0j_0k}T_{\bdu}^{-2}, & \text{ if $(i,j,k)\in \calZ_4(\calA_{\bdu})$,}\\
            T_{ijk}-T_{ij_0k}T_{i_0jk_0}T_{\bdu}^{-1}, & \text{ if $(i,j,k)\in \calZ_5^{\ge 2}(\calA_{\bdu})$,}\\
            T_{ijk}, & \text{ otherwise.}
        \end{cases}
    \]

    Take arbitrary $(i,j,k)\in \calZ_2^{\ge 2}(\calA_{\bdu})$, we have
    $g_1\coloneqq T_{\bdu}T_{ijk}-T_{ijk_0}T_{i_0j_0k}\in
    \bfI_2(\calA_{\bdu})$. Notice that
    $\varphi(g_1)=T_{\bdu}(T_{ijk}+T_{ijk_0}T_{i_0j_0k}T_{\bdu}^{-1})-T_{ijk_0}T_{i_0j_0k}=T_{\bdu}T_{ijk}$.
    So $T_{ijk}\in \varphi(\bfI_2(\calA_{\bdu}))$.

    If we take the ideal
    \begin{align*}
        \frakb'\coloneqq  \Big(T_{ijk}: (i,j,k)\in \calZ_2^{\ge
            2}(\calA_{\bdu})\cup \calZ_4(\calA_{\bdu})\cup \calZ_5^{\ge
            2}(\calA_{\bdu})\Big)
    \end{align*}
    in $\KK[\bfT_{\calA_{\bdu}}][T_{\bdu}^{-1}]$, then by similar arguments, we
    have $\frakb'\subseteq \varphi(\bfI_2(\calA_{\bdu}))$.

    Recall that in the proof of \Cref{VD-InductionOrder}, we defined
    \begin{align*}
        \calE\coloneqq\calA_{\bdu}^{+}\setminus \left(\calZ_2^{\ge 2}(\calA_{\bdu})\cup
            \calZ_4(\calA_{\bdu})\cup \calZ_5^{\ge 2}(\calA_{\bdu})\right),
        \intertext{and} 
        \calH\coloneqq\calE\setminus\Set{(1,j_0,k): k_0<k\le \min
            (\gamma,c_{\calD^{\ge 2}})},
    \end{align*} 
    where $\gamma\coloneqq c_{\calD'}(\bdu)$.  It is clear that
    $\frakb'+\bfI_2(\calH)\subseteq \varphi(\bfI_2(\calA_{\bdu}))$. We claim
    that they are equal:
    \[
    \frakb'+\bfI_2(\calH)= \varphi(\bfI_2(\calA_{\bdu})).
    \]

    Notice that
    \begin{itemize}
        \item if $(i,j,k)\in \calZ_2^{\ge 2}(\calA_{\bdu})$, then $\varphi(T_{ijk})\equiv T_{ijk_0}T_{i_0j_0k}T_{\bdu}^{-1} \mod \frakb'$;
        \item if $(i,j,k)\in \calZ_5^{\ge 2}(\calA_{\bdu})$, then $ \varphi(T_{ijk})\equiv T_{ij_0k}T_{i_0jk_0}T_{\bdu}^{-1} \mod \frakb'$;
        \item if $(i,j,k)\in \calZ_4(\calA_{\bdu})$, then  
            \begin{align*}
                \varphi(T_{ijk})
                &\equiv  \varphi(T_{ijk_0}T_{i_0j_0k})T_{\bdu}^{-1}
                \equiv \varphi(T_{ij_0k}T_{i_0jk_0})T_{\bdu}^{-1}\\
                &\equiv  \varphi(T_{i_0jk}T_{ij_0k_0})T_{\bdu}^{-1}
                \equiv T_{ij_0k_0}T_{i_0jk_0}T_{i_0j_0k}T_{\bdu}^{-2}
                \mod \frakb';
            \end{align*}
        \item $\bfI_2(\calE)=\bfI_2(\calH)$.
    \end{itemize}

    Roughly speaking, the reduction by $\frakb'$ has the effect of projecting
    the points in the above three designated zones to the coordinate planes and
    axes centered at $\bdu$, where the landing points lay in $\calH$. Thus,
    when dealing with the $2$-minors, in the following situation
    \begin{itemize}
        \item[(\dag-1)] the points of both columns of the underlying $2\times 2$ matrix
            belong to $\frakb'$, or 
        \item[(\dag-2)] the points of one column belong to $\frakb'$, while the points in
            the other belong to $\calH$,
    \end{itemize}
    then, we should be fine after factoring common factors; we will explain
    this by an example later in (a)-(iii).  Thus, to show the above claim of equality, it
    suffices to consider the irregular generators $\bfI_{2,*}(\bdv_1,\bdv_2)$.
    Say, $\bdv_1=(i_1,j_1,k_1)$ and $\bdv_2=(i_2,j_2,k_2)$. Since at this stage
    $c_{\calD^{\ge 2}}\ge k_0$, by the projection property, we have
    $\calZ_6^1(\calD)=\calZ_6(\calD)$.

    \begin{enumerate}[a]
        \item We consider the case when $*=x$.

            \begin{enumerate}[i]
                \item We first investigate $2$-minors that involve
                    $\bdu=\bdv_1$. Notice that if $\bdu$ can exchange
                    $x$-coordinates with $\bdv_2$, then 
                    \[
                        \bdv_2\in \calZ_4^{\ge
                            2}\cup \calZ_5^{\ge 2}\cup \Set{(1,j_0,k)|k_0<k\le
                            \min(\gamma,c_{\calD^{\ge 2}})}.
                    \]  
                    After reductions by
                    $\frakb'$, the only irregular $2$-minors involving
                    $T_{\bdu}$ take the form $g_2\coloneqq
                    T_{\bdu}T_{\bdv}-T_{i_2j_0k_0}T_{i_0j_2k_2}$ with
                    $\bdv_2\in \calZ_5^{\ge 2}(\calA_{\bdu})$.  We have 
                    \begin{align*} 
                        \varphi(g_2)& \equiv
                        T_{\bdu}(T_{i_2j_0k_2}T_{i_0j_2k_0}T_{\bdu}^{-1})-T_{i_2j_0k_0}T_{i_0j_2k_2}\mod \frakb'\\
                        &= T_{i_2j_0k_2}T_{i_0j_2k_0}-T_{i_2j_0k_0}T_{i_0j_2k_2} ,
                    \end{align*}
                    and $T_{i_2j_0k_2}T_{i_0j_2k_0}-T_{i_2j_0k_0}T_{i_0j_2k_2}\in \bfI_2(\calH)$.

                \item The other irregular case is, by symmetry, when $\bdv_1\in
                    \calZ_5^{\ge 2}(\calA)$, $\bdv_2\in \calZ_5^1(\calA)$.
                    Thus, $i_1\ge 2$ and $i_2=1=i_0$. By symmetry, we may also
                    assume that $j_1\le j_2$.  Now,
                    $\bfI_{2,x}(\bdv_1,\bdv_2)=T_{\bdv_1}T_{\bdv_2}-T_{i_2j_1k_1}T_{i_1j_2k_2}$.
                    Therefore,  
                    \[
                        \varphi(\bfI_{2,x}(\bdv_1,\bdv_2)) \equiv
                        T_{\bdu}^{-1}(T_{i_1j_0k_1}T_{i_0j_1k_0}T_{i_2j_2k_2}
                        -T_{i_2j_1k_1}T_{i_1j_0k_2}T_{i_0j_2k_0}) \mod \frakb'.
                    \]
                    However,
                    \begin{align*}
                         & T_{i_1j_0k_1}T_{i_0j_1k_0}T_{i_2j_2k_2} -T_{i_2j_1k_1}T_{i_1j_0k_2}T_{i_0j_2k_0} \\
                         &\qquad = T_{i_0j_1k_0}(T_{i_1j_0k_1}T_{i_2j_2k_2}-T_{i_1j_0k_2}T_{i_2j_2k_1})\\
                         &\qquad + T_{i_1j_0k_2}(T_{i_0j_1k_0}T_{i_2j_2k_1}-T_{i_2j_1k_1}T_{i_0j_2k_0})
                         \in \bfI_2(\calH).
                    \end{align*}
                    Notice that $T_{i_2j_2k_1}$ exists with $(i_2,j_2,k_1)\in \calZ_5^1$.

                \item All other $2$-minors are regular in the sense of (\dag) and can be reduced by $\frakb'$ to $2$-minors in $\bfI_2(\calH)$. 

                    We first look at an example in (\dag-1). Say $\bdv_1\in \calZ_2^{\ge 2}$ while $\bdv_2\in \calZ_5^{\ge 2}$. Suppose that $\bdv_1$ can exchange $x$-coordinates with $\bdv_2$. Now,
                    \begin{align*}
                        \varphi(\bfI_{2,x}(\bdv_1,\bdv_2)) &=\varphi(T_{i_1j_1k_1}T_{i_2j_2k_2}-T_{i_2j_1k_1}T_{i_1j_2k_2}) \\
                        &\equiv (T_{i_1j_1k_0}T_{i_0j_0k_1}T_{\bdu}^{-1})(T_{i_2j_0k_2}T_{i_0j_2k_0}T_{\bdu}^{-1})\\
                        &\qquad -(T_{i_2j_1k_0}T_{i_0j_0k_1}T_{\bdu}^{-1})(T_{i_1j_0k_2}T_{i_0j_2k_0}T_{\bdu}^{-1}) \mod \frakb'\\
                        &= T_{\bdu}^{-2}T_{i_0j_0k_1}T_{i_0j_2k_0}(T_{i_1j_1k_0}T_{i_2j_0k_2}- T_{i_2j_1k_0}T_{i_1j_0k_2}) \in \bfI_2(\calH).
                    \end{align*}
                    Therefore, still,  $\varphi(\bfI_{2,x}(\bdv_1,\bdv_2))\in  \frakb'+\bfI_2(\calH)$.

                    We may also look at one example in (\dag-2). Say $\bdv_1\in \calZ_5^1$ while $\bdv_2\in \calZ_4^{\ge 2}$. If $\bdv_1$ can exchange $x$-coordinates with $\bdv_2$, then $(i_2,j_1,k_1),(i_1,j_2,k_2)\in \calA_{\bdu}$. Now, since $i_1=1=i_0$, 
                    \begin{align*}
                        \varphi(\bfI_{2,x}(\bdv_1,\bdv_2)) &=\varphi(T_{i_1j_1k_1}T_{i_2j_2k_2}-T_{i_2j_1k_1}T_{i_1j_2k_2}) \\
                        &\equiv T_{i_1j_1k_1} (T_{i_0j_0k_2}T_{i_0j_2k_0}T_{i_2j_0k_0}T_{\bdu}^{-2})\\
                        &\qquad-(T_{i_2j_0k_1}T_{i_0j_1k_0}T_{\bdu}^{-1})(T_{i_1j_2k_0}T_{i_0j_0k_2}T_{\bdu}^{-1}) \mod \frakb'\\
                        &=(T_{\bdu}^{-1}T_{i_0j_0k_2})T_{\bdu}^{-1}(T_{i_1j_1k_1} T_{i_0j_2k_0}T_{i_2j_0k_0}-T_{i_2j_0k_1}T_{i_0j_1k_0}T_{i_1j_2k_0})\\
                        &\equiv (T_{\bdu}^{-1}T_{i_0j_0k_2}) \varphi(\bfI_{2,x}( \bdv_1,(i_2,j_2,k_0))) \mod \frakb'.
                    \end{align*}
                    Therefore, by (ii), $\varphi(\bfI_{2,x}(\bdv_1,\bdv_2))\in  \frakb'+\bfI_2(\calH)$.  
             
                    All other cases are similar, hence omitted.
            \end{enumerate}

        \item For the case when $*=y$ or $z$, all $2$-minors are like in (a)-(iii), and can similarly be reduced by $\frakb'$ to $2$-minors in $\bfI_2(\calH)$.
    \end{enumerate}

    Thus, we have shown that
    $\frakb'+\bfI_2(\calH)=\varphi(\bfI_2(\calA_{\bdu}))$. Consequently,
    it suffices to show that $\bfI_2(\calH)\subseteq \KK[\bfT_{\calH}]$ is a
    prime ideal.  Notice that we have seen in \Cref{calculation-lemma} that
    $\calH=\widetilde{\calA}_{(1,j_0+1,1)}$, which is a smaller case. Thus, by
    induction, the primeness of $\bfI_2(\calH)$ is guaranteed.

    \bigskip
    \textbf{Second stage}

    After the first stage, we are dealing with the case where we removed all the points in $\calC$.  Now we flip $\calD$ to get $\calS(\calD)$, again, written as $\calD'$. Using the notation in the proof of Theorem \ref{VD-InductionOrder}, the current case is $\calS(\calD\setminus\calC)=\calD'_{(1,c_{\calD^{\ge 2}}+1,1)}$.  Similar to the first stage, we will prove by induction on removing lexicographically initial points in the $x=1$ layer of $\calD'_{(1,c_{\calD^{\ge 2}}+1,1)}$. The minimal case will be when we remove all the $x=1$ points and this is settled by induction. Notice that in the proof of the first stage, the ``ceiling restriction'' of choosing points $(1,j,k)$ with $k\le c_{\calD^{\ge 2}}$ is only used to ensure that $\calZ_6=\calZ_6^1$.  In the current case of $\calD'_{(1,c_{\calD^{\ge 2}}+1,1)}$, for any point $\bdu$ in the $x=1$ layer, we will automatically get $\calZ_6=\calZ_6^1$. So the proof is similar and easier. 
\end{proof}

\section{Blowup algebras}

It is time for the main theorems of this work. Indeed, we show that the ideal $\bfI_2(\calD)$ is the presentation ideal of the special fiber ring $\calF(I_{\calD})$. Since $\bfI_2(\calD)$ has nice properties, so does the special fiber ring $\calF(I_{\calD})$. Moreover, we can extend the result to the Rees algebra $\calR(I_{\calD})$ easily because the ideal $I_{\calD}$ satisfies $\ell$-exchange property (\Cref{L-exchange-def}).  

\begin{Theorem}
    \label{ToricIsKoszul} 
    Let $\calD$ be a three-dimensional Ferrers diagram which satisfies the projection property.  Then the special fiber ring $\calF(I_{\calD})$ is a Koszul Cohen--Macaulay normal domain.
\end{Theorem}

\begin{proof}
    Notice that $\bfI_2(\calD)\subseteq J_\calD$. Since these two homogeneous ideals are prime and have the same codimension by \Cref{I2prime}, \Cref{DimI2} and \Cref{DimFiber}, they actually coincide. The Cohen--Macaulay property follows from \Cref{VD-InductionOrder}. Since $\bfI_2(\calD)$ has a squarefree initial ideal, the normal property follows from \cite[Theorem 5.16]{MR2850142}.  The Koszul property follows from \cite[Theorem 6.7]{MR2850142}.
\end{proof}

Next, we consider the Rees algebra $\calR(I_{\calD})$ of $I_{\calD}$. The
strategy is similar to that in \cite[Section 6]{arXiv:1706.07462}.

\begin{Definition}
    [{\cite[Definition 4.1]{MR2195995}}]
    \label{L-exchange-def}
    Let $I=(f_1,\dots,f_m)\subset R=\KK[x_1,\dots,x_n]$ be a monomial ideal
    generated in one degree.  Let $\KK[\bfT]\coloneqq \KK[T_1,\dots,T_m]$  and
    $J$ be the toric ideal of $I$, i.e., the kernel of the surjective
    homomorphism 
    \[
        \psi:\KK[\bfT]\to \KK[f_1,\dots,f_m],
    \]
    defined by $\psi(T_i)=f_i$ for all $i$. Let $<$ be a monomial order on
    $\KK[\bfT]$. A monomial $\bfT^{\bda}$ in $\KK[\bfT]$ is called a
    \emph{standard monomial} of $J$ with respect to $<$, if it does not belong
    to the initial ideal of $J$.

    The monomial ideal $I$ satisfies the \emph{$\ell$-exchange property} with
    respect to the monomial order $<$ on $\KK[\bfT]$, if the following
    condition is satisfied: let $\bfT^{\bda}$ and $\bfT^{\bdb}$ be any two
    standard monomials of $J$ with respect to $<$ of the same degree, with
    $u=\psi(\bfT^{\bda})$ and $v=\psi(\bfT^{\bdb})$ satisfying
    \begin{enumerate}[i]
        \item $\deg_{x_t}(u)=\deg_{x_t}(v)$ for $t=1,\dots,q-1$ with $q\le n-1$,
        \item $\deg_{x_q}(u)<\deg_{x_q}(v)$.
    \end{enumerate}
    Then there exists an integer $k$, and an integer $q<j\le n$ such that
    $x_qf_k/x_j\in I$.
\end{Definition}

Similar to \cite[Example 4.2]{MR2195995}, we have

\begin{Lemma}
    \label{LEx}
    Let $\calD$ be a three-dimensional Ferrers diagram. Then the Ferrers ideal
    $I_{\calD} \subset R=\KK[x_1,\dots,x_m,y_1,\dots,y_n,z_1,\dots,z_p]$
    satisfies the $\ell$-exchange property with respect to any monomial order
    $<$ on $\KK[\bfT]$.
\end{Lemma}

\begin{proof}
    We will use the notation in the previous definition.  Without loss of
    generality, we may assume that $\deg_{x_t}(u)=\deg_{x_t}(v)$ for
    $t=1,\dots,q-1$ with $q\le m$ and $\deg_{x_q}(u)<\deg_{x_q}(v)$.  Since
    \[
        3\sum_{i=1}^m \deg_{x_i}(u) =\deg(u)=\deg(v)=3\sum_{i=1}^m
        \deg_{x_i}(v),
    \]
    we can see indeed that $q\le m-1$. Thus, we can find some $f_{\delta}$ and
    $q<j\le m$ with $\deg_{x_j}(f_{\delta})\ge 1$. Notice that
    $f_{\delta}=x_jy_*z_*$. Thus, $x_qf_\delta/x_j\in I_{\calD}$, since
    ${\calD}$ is a Ferrers diagram.
\end{proof}

The crucial weapon for our final result is the following.

\begin{Lemma}
    [{\cite[Theorem 5.1]{MR2195995}}]
    \label{ReesIdeal}
    Let $I=(f_1,\dots,f_m)\subset R=\KK[x_1,\dots,x_n]$ be a monomial ideal generated in one degree, satisfying the $\ell$-exchange property.  Let $<_{lex}$ be the lexicographic order on $R$ with respect to $x_1>\cdots > x_n$ and $<'$ an arbitrary monomial order on $\bfT$. Let $<_{lex}'$ be the product order of $<'$ and $<_{lex}$.  Then the reduced Gr\"obner basis of the Rees ideal of $I$ with respect to $<_{lex}'$ consists of all binomials belonging to the reduced Gr\"obner basis of $J$ with respect to $<'$ together with the binomials $x_iT_k-x_jT_l$, where $x_i>x_j$ with $x_if_k=x_jf_l$ and $x_j$ is the smallest variable for which $x_if_k/x_j$ belongs to $I$. In particular, $I$ is of fiber type.
\end{Lemma}

As an application, we have

\begin{Theorem}
    Let $\calD$ be a three-dimensional Ferrers diagram which satisfies the
    projection property.  Then the Rees algebra $\calR(I_{\calD})$ is Koszul and
    the ideal $I_{\calD}$ is of fiber type.
\end{Theorem}

\begin{proof}
    By \Cref{ToricIsKoszul}, we know the toric ideal $J_{\calD}$ of $I_{\calD}$
    has a quadratic Gr\"obner basis. Hence by \Cref{LEx} and \Cref{ReesIdeal},
    $I_{\calD}$ is of fiber type and the Rees ideal of $I_{\calD}$ has a
    quadratic Gr\"obner basis. It follows from \cite[Theorem 6.7]{MR2850142}
    that $\calR(I_{\calD})$ is Koszul.
\end{proof}

We close with some questions for future research.

\begin{Question}
    \label{deg3-conj}
    Let $\calD$ be a three-dimensional Ferrers diagram. Is the degree of
    minimal binomial generators of the special fiber ideal at most three?
    Is the special fiber ideal always Cohen--Macaulay or normal? What
    about the Rees algebra?
\end{Question}

\begin{acknowledgment*}
    The authors thank Uwe Nagel and Paolo Mantero for the discussion of the
    early version of this manuscript.  They also thank the anonymous referee
    for the careful reading and helpful suggestions that greatly improved the
    presentation of this paper.  The second author is partially supported by
    the ``Fundamental Research Funds for the Central Universities''.
\end{acknowledgment*}


\end{document}